\documentclass[a4paper,twoside,10pt]{article}
	\usepackage{a4wide}
	\usepackage{graphicx}
	\usepackage{xcolor}
	\usepackage{amsmath,amssymb,amsthm,amsfonts,mathtools,amsbsy}
	\usepackage{todonotes}
	\usepackage{orcidlink}
	\usepackage{hyperref}
	\hypersetup{colorlinks=true,linkcolor=blue!60!black,filecolor=purple!50!black,citecolor=blue!50!green}
	\usepackage[noabbrev,nameinlink,capitalise]{cleveref}
	\usepackage{tikz}
	\usetikzlibrary{positioning,calc,decorations.pathreplacing}
	\usepackage{caption}
	\usepackage{enumitem}

\newenvironment{eqs} %
 { \begin{equation} \begin{aligned} } %
 { \end{aligned} \end{equation} \ignorespacesafterend } %

\makeatletter
\newtheorem{theorem}{Theorem}[section]
\newtheorem{lemma}[theorem]{Lemma}

\newtheorem{corollary}[theorem]{Corollary}

\newtheorem{remark}[theorem]{Remark}
\crefname{assumption}{Assumption}{Assumptions}

\makeatother
\newtheorem*{remark*}{Remark}

\newtheorem*{acknowledgment*}{Acknowledgment}

\newcommand{\CA}{C_A}
\newcommand{\Ca}{C_a}
\newcommand{\ca}{c_a}
\newcommand{\cA}{c_A}
\newcommand{\cAP}{c_{A_0}}

\newcommand{\const}[1]{C_{\eqref{#1}}}

\newcommand{\bpm}{\begin{pmatrix}}
\newcommand{\epm}{\end{pmatrix}}

\newcommand{\Vsh}{V_{\!\star \hspace*{-0.2mm}h}}

\newcommand{\dgsnorm}[1]{\norm{#1}_{\Vsh}}
\newcommand*{\abs}[1]{\left|#1\right|}

\newcommand{\jmp}[1]{[\![ #1 ]\!]}

\newcommand{\AK}{{A_K}}
\newcommand{\AKO}{{A_{\tilde K,0}}}
\newcommand{\Lh}{\mathbb{L}_h}

\newcommand{\LKh}{\mathbb{L}_h(K)}

\DeclareMathOperator{\Div}{div}

\DeclareMathOperator{\grad}{\nabla}

\DeclareMathOperator{\Lap}{\Delta}

\renewcommand{\dim}{\operatorname{dim}}

\newcommand{\eps}{\varepsilon}
\newcommand{\inner}[1]{( #1 )}

\newcommand*{\norm}[1]{\left\|#1\right\|}

\newcommand\restr[2]{{ \left.\kern-\nulldelimiterspace #1 \vphantom{\big|} \right|_{#2} }}

\newcommand{\IL}{\mathbb{L}}

\newcommand{\IP}{\mathbb{P}}

\newcommand{\IR}{\mathbb{R}}

\newcommand{\IT}{\mathbb{T}}

\newcommand{\calD}{\mathcal{D}}
\newcommand{\calE}{\mathcal{E}}

\newcommand{\calI}{\mathcal{I}}

\newcommand{\calK}{\mathcal{K}}

\newcommand{\bn}{\mathbf{n}}

\newcommand{\bw}{\boldsymbol{w}}
\newcommand{\bx}{\mathbf{x}}

\definecolor{pscol}{rgb}{0.8,0,0}

\usepackage{pgfplots} 
\pgfplotsset{compat=newest}
\usepgfplotslibrary{colorbrewer,groupplots}
\pgfplotsset{
	discard if not/.style 2 args={
	x filter/.append code={
	\edef\tempa{\thisrow{#1}}
	\edef\tempb{#2}
	\ifx\tempa\tempb
	\else
	
	\fi
	}
	}
}
\usepackage{pgfplotstable} 
\usepackage{booktabs} 
\usepackage{colortbl}
\makeatletter
\pgfplotstableset{
	discard if not/.style 2 args={
	row predicate/.append code={
	\def\pgfplotstable@loc@TMPd{\pgfplotstablegetelem{##1}{#1}\of}
	\expandafter\pgfplotstable@loc@TMPd\pgfplotstablename
	\edef\tempa{\pgfplotsretval}
	\edef\tempb{#2}
	\ifx\tempa\tempb
	\else
	\pgfplotstableuserowfalse
	\fi
	}
	}
}
\makeatother

\pgfplotscreateplotcyclelist{paulcolors1}{%
	orange,every mark/.append style={solid,fill=orange},mark=square*,very thick,mark size=3pt\\%
	teal,every mark/.append style={solid,fill=teal},mark=square*,very thick,mark size=3pt\\%
    violet,every mark/.append style={solid,fill=violet},mark=square*,very thick,mark size=3pt\\%
}
\pgfplotscreateplotcyclelist{paulcolors2}{%
	orange,every mark/.append style={solid,fill=orange},mark=square*,very thick,mark size=3pt\\%
	orange,densely dashed,every mark/.append style={solid,fill=orange},mark=*,very thick,mark size=3pt\\%
	teal,every mark/.append style={solid,fill=teal},mark=square*,very thick,mark size=3pt\\%
	teal,densely dashed,every mark/.append style={solid,fill=teal},mark=*,very thick,mark size=3pt\\%
    violet,every mark/.append style={solid,fill=violet},mark=square*,very thick,mark size=3pt\\%
    violet,densely dashed,every mark/.append style={solid,fill=violet},mark=*,very thick,mark size=3pt\\%
}

\pgfplotscreateplotcyclelist{paulcolors3}{%
	orange,every mark/.append style={solid,fill=orange},mark=square*,very thick,mark size=3pt\\%
	orange,densely dashed,every mark/.append style={solid,fill=orange},mark=*,very thick,mark size=3pt\\%
	orange,dotted,every mark/.append style={solid,fill=orange},mark=diamond*,very thick,mark size=3pt\\%
	teal,every mark/.append style={solid,fill=teal},mark=square*,very thick,mark size=3pt\\%
	teal,densely dashed,every mark/.append style={solid,fill=teal},mark=*,very thick,mark size=3pt\\%
	teal,dotted,every mark/.append style={solid,fill=teal},mark=diamond*,very thick,mark size=3pt\\%
	violet,every mark/.append style={solid,fill=violet},mark=square*,very thick,mark size=3pt\\%
	violet,densely dashed,every mark/.append style={solid,fill=violet},mark=*,very thick,mark size=3pt\\%
	violet,dotted,every mark/.append style={solid,fill=violet},mark=diamond*,very thick,mark size=3pt\\%
}

	\usepackage{tabularray}
	\UseTblrLibrary{booktabs}
	
	\usepackage{orcidlink}
	
	\usepackage{tcolorbox}
	\usepackage{pgf}
	\usepackage{pgfplots} 
	\numberwithin{equation}{section}

	\definecolor{cborange}{HTML}{ff7f0e}
	\definecolor{cbblue}{HTML}{1f77b4}
	\definecolor{cbgreen}{HTML}{2ca02c}
	\pgfplotscreateplotcyclelist{colorshp}{%
		cbblue,every mark/.append style={solid,fill=cbblue},mark=square*,very thick,mark size=3pt\\%
		cborange, every mark/.append style={solid,fill=cborange},mark=*,very thick,mark size=3pt\\%
		cbgreen,every mark/.append style={solid,fill=cbgreen},mark=square*,very thick,mark size=3pt\\%
		purple, every mark/.append style={solid,fill=purple},mark=*,very thick,mark size=3pt\\%
	}
	
	\pgfplotscreateplotcyclelist{colorsh}{%
		cbblue,every mark/.append style={solid,fill=cbblue},mark=square*,very thick,mark size=3pt\\%
		cborange, every mark/.append style={solid,fill=cborange},mark=*,very thick,mark size=3pt\\%
	cbblue,every mark/.append style={solid,fill=cbblue},mark=square*,very thick,mark size=3pt\\%
		cborange, every mark/.append style={solid,fill=cborange},mark=*,very thick,mark size=3pt\\%
	}

	\pgfplotsset{
		discard if not/.style 2 args={
			x filter/.append code={
				\edef\tempa{\thisrow{#1}}
				\edef\tempb{#2}
				\ifx\tempa\tempb
				\else
				
				\fi
			}
		}
	}
	\pgfplotsset{
		discard if/.style 2 args={
			x filter/.append code={
				\edef\tempa{\thisrow{#1}}
				\edef\tempb{#2}
				\ifx\tempa\tempb
				
				\fi
			}
		},}   
	\pgfplotsset{compat=1.16}%
	
	\newcommand{\CycleNextGruoupPloth}[2]
	{\nextgroupplot[ylabel={$\norm{u-u_h}_{\Vh}$},xlabel={$\sqrt{\mathrm{N}_{\mathrm{DoFs}}}$},title={$p=#1$}, title style={font=\large},
		label style={font=\large},
		tick label style={font=\large},
		legend style={font=\large}]
		\addplot+[discard if not={p}{#1},discard if not={method}{dg}] table [x=ndof12, y=dgerror, col sep=comma] {results/example1.csv};
		\addplot+[discard if not={p}{#1},discard if not={method}{embt}] table [x=ndof12, y=dgerror, col sep=comma] {results/example1.csv};
		\addplot+[discard if not={p}{#1},discard if not={method}{dg},discard if={#2}{0}, only marks, mark size=2.5pt,
	visualization depends on=\thisrow{#2} \as \labela,
	nodes near coords=\pgfmathprintnumber{\labela}
	,
	every node near coord/.append style={
	text=cbblue,
	font=\bfseries\boldmath\small,
		inner sep=1pt,
		xshift=0.6ex,
		yshift=4.1ex,
		scale=1,/pgf/number format/fixed,
		/pgf/number format/precision=2,/pgf/number format/fixed zerofill}
	] table [x=ndof12, y=dgerror, col sep=comma] {results/example1.csv};
	\addplot+[discard if not={p}{#1},discard if not={method}{embt},discard if={#2}{0}, only marks,
	visualization depends on=\thisrow{#2} \as \labela,
	nodes near coords=\pgfmathprintnumber{\labela}
	,
	every node near coord/.append style={
	text=cborange,
	font=\bfseries\boldmath\small,
		inner sep=1pt,
		xshift=-5.5ex,
		yshift=2ex,
		scale=1,/pgf/number format/fixed,
		/pgf/number format/precision=2,/pgf/number format/fixed zerofill}
	] table [x=ndof12, y=dgerror, col sep=comma] {results/example1.csv};
	\legend{$\Vh$, $\IT$}
	}
	
	\renewcommand{\bx}{\boldsymbol{x}}

	\newcommand{\R}{\mathbb{R}}
	\newcommand{\N}{\mathbb{N}}
	\newcommand{\Vh}{\mathbb{V}_{\! h}}
	\newcommand{\ITh}{\mathbb{T}_{\mspace{-1.5mu} h}}
	
	\renewcommand{\Vsh}{\mathbb{V}_{*h}}
	\newcommand{\VK}{\mathbb{V}_h(K)}
	\newcommand{\uh}{u_h}
	\newcommand{\vh}{v_h}
	
	\newcommand{\hperp}{h_{\perp}}
	\newcommand{\FK}{F_K}
	\newcommand{\GK}{G_K}
	\newcommand{\DK}{D_K}
	\newcommand{\dK}{d_K}
	\newcommand{\Kt}{\widetilde{K}}
	\newcommand{\Kh}{\widehat{K}}
	\newcommand{\hKot}{h_{1,K}}
	\newcommand{\hKtt}{h_{2,K}}
	\newcommand{\xt}{\widetilde{x}}

	\newcommand{\mvl}[1]{\{\!\!\{#1\}\!\!\}}
	\newcommand{\Norm}[2]{\|#1\|_{#2}}
	\newcommand{\ah}{a_h}
	\newcommand{\lh}{\ell_h}
	
	\usepackage{soul}

\begin{document}
	
	\title{Embedded Trefftz DG method for reaction--diffusion problems on anisotropic meshes}
	
	\author{Sergio G\'omez\thanks{Department of Mathematics and Applications, University of Milano--Bicocca, 20125 Milan, Italy \mbox{(\href{mailto:sergio.gomezmacias@unimib.it}{sergio.gomezmacias@unimib.it})} } \thanks{IMATI-CNR ``Enrico Magenes", Via Ferrata 5, 27100, Pavia, Italy}
		\and 
		Chiara Perinati\thanks{Department of Mathematics, University of Pavia, 27100 Pavia, Italy (\href{mailto:chiara.perinati01@universitadipavia.it}{chiara.perinati01@universitadipavia.it})}
		\and 
		Paul Stocker\thanks{Faculty of Mathematics, University of Vienna, 1090 Vienna, Austria (\href{mailto:paul.stocker@univie.ac.at}{paul.stocker@univie.ac.at})}
		\and	
		Igor Voulis\thanks{Institut für Numerische und Angewandte Mathematik, Georg-August-Universität Göttingen, Germany (\href{mailto:i.voulis@math.uni-goettingen.de}{i.voulis@math.uni-goettingen.de})} 
	}
	\date{}
	
	\maketitle

	\begin{abstract}
	\noindent
	We present and analyze an embedded Trefftz discontinuous Galerkin method for reaction--diffusion problems on anisotropic meshes. 
	The method is constructed by imposing a relaxed local Trefftz condition via an embedding into a tensor-product DG space, yielding a reduced global system while preserving the approximation properties of the underlying high-order discretization.
	We prove 
	stability and quasi-optimality on anisotropic, possibly curved, quadrilateral elements, and derive anisotropic a priori error estimates.
	Numerical experiments for $h$- and $hp$-refinement, including curved-domain examples, validate the theoretical results.
	\end{abstract}

	\section{Introduction}\label{sect::introduction}
	
		Anisotropic mesh refinement is a classical tool for resolving directional features, layers, and strongly nonuniform solution behavior,
        and its analysis is well developed
		for conforming finite element methods, see, e.g.,~\cite{apel99,AD92ai,AL96am,FP03ae},
        as well as for discontinuous Galerkin (DG) methods, see~\cite{G03dm,G06hp,GHH07hp,GHH09ap,GHH08ar}.
		
		DG methods provide a flexible framework for the design of high-order methods on general meshes.
		Trefftz DG methods 
		 reduce the number of degrees of freedom by using local approximation spaces that are piecewise contained in the kernel of the underlying differential operator.
		Trefftz DG methods have been widely studied in the context of wave propagation problems; cf.~\cite{CHP19ha,HMP14lr,HMP16sv,KSTW14dg}.
		For more general elliptic partial differential equations (PDEs), however, the construction of suitable Trefftz-type spaces is much less straightforward.
		Some recent progress in this direction has been achieved by quasi-Trefftz techniques; see, e.g.,~\cite{IMPS25qe,P23qd}.
		
		The \emph{embedded Trefftz} approach, introduced in \cite{LS_IJMNE_2023,lozinski19}, overcomes this difficulty in a systematic way by constructing a suitable embedding, which embeds a subspace with a relaxed Trefftz property into a standard DG space.
		It thereby provides a general framework for Trefftz-type DG methods that remains effective even when the kernel of the PDE operator is not explicitly known, or is trivial.
        Such a technique has been successfully applied to several classes of PDEs, including elliptic problems, wave propagation, and fluid dynamics problems; see \cite{GPS_JSC_2025,H24,HLSW23ut,LLS_NM_2024,LLSV_ARXIV_2024,LS_IJMNE_2023,S23,SV_ARXIV_2026}.
		
		\paragraph{Motivation.}
		To motivate the use of the proposed embedded Trefftz DG method on anisotropic quadrilateral and hexahedral meshes, we compare the resulting global algebraic systems with those of standard tensor-product DG and hybridizable DG (HDG) methods.
		For HDG, all element-interior degrees of freedom are condensed, so that only facet degrees of freedom remain globally coupled.
		As a rough proxy for computational cost, Table~\ref{tab:ndof-nnze} reports the number of globally coupled degrees of freedom and nonzero matrix entries per element.
		The results show that embedded Trefftz discretizations are competitive with HDG in two space dimensions and yield the smallest global systems in three space dimensions.
		This highlights their potential for reducing global algebraic complexity while retaining approximation properties; see \cite{LS_IJMNE_2023,LSZ_PAMM_2024} for further comparisons.
		\newcommand{\ncdof}{\texttt{ncdof}}
		\newcommand{\nnze}{\texttt{nnze}}
		\begin{table}[!ht]
			    \centering
			    \begin{tblr}{
  colspec = {Q[l]| *{5}{Q[r]} | *{5}{Q[r]}},
  row{4,7} = {bg=gray!10},
  row{2} = {font=\bfseries},
}
\toprule
& \SetCell[c=5]{c}{\textbf{2D: quadrilateral meshes}}  & & & & & \SetCell[c=5]{c}{\textbf{3D: hexahedral meshes}}\\
$\text{method} \downarrow~\!\!\setminus~\!
p\!\rightarrow$ \!\!\! & 2 & 3 & 4 & 5 & 6 & 2 & 3 & 4 & 5 & 6\\
\midrule
$\texttt{ncdof}_{\text{DG}}/N_{\text{El}}$\!\! & \!9 & \!16 & \!25 & \!36 & \!49 & \!27 & \!64 & \!125 & \!216 & \!343\\
$\texttt{ncdof}_{\text{TDG}}/N_{\text{El}}$\!\! & \!5 & \!7 & \!9 & \!11 & \!13 & \!9 & \!16 & \!25 & \!36 & \!49\\
$\texttt{ncdof}_{\text{HDG}}/N_{\text{El}}$\!\! & \!6 & \!8 & \!10 & \!12 & \!14 & \!27 & \!48 & \!75 & \!108 & \!147\\
\midrule
$\texttt{nnze}_{\text{DG}}/N_{\text{El}}$ & \!405 & \!1280 & \!3125 & \!6480 & \!12005 & \!5103 & \!28672 & \!109375 & \!326592 & \!823543\\
$\texttt{nnze}_{\text{TDG}}/N_{\text{El}}$ & \!125 & \!245 & \!405 & \!605 & \!845 & \!567 & \!1792 & \!4375 & \!9072 & \!16807\\
$\texttt{nnze}_{\text{HDG}}/N_{\text{El}}$ & \!126 & \!224 & \!350 & \!504 & \!686 & \!2673 & \!8448 & \!20625 & \!42768 & \!79233\\
\bottomrule
\end{tblr}

			    \caption{Number of globally coupled degrees of freedom $\ncdof$ and nonzero matrix entries $\nnze$ per element for tensor-product DG, embedded Trefftz DG, and HDG, with polynomial degree $p$.
}
			    \label{tab:ndof-nnze}
			\end{table}   
		\vspace{-2em}
		
		\paragraph{Contributions.}
		In this work, we study an embedded Trefftz discontinuous Galerkin method for singularly perturbed reaction--diffusion problems on anisotropic quadrilateral meshes.
		The method combines the approximation advantages of anisotropic DG discretizations with the reduction of globally coupled degrees of freedom provided by the embedded Trefftz construction.
		
		The main novelty is the design of a \emph{local test space compatible with tensor-product polynomial bases}.
		This yields stability of the embedded Trefftz formulation in the anisotropic setting   and 
		provides a \emph{first systematic route to embedded Trefftz methods formulated in tensor-product polynomial spaces}, rather than in standard (simplicial) polynomial bases, thereby 
		opening the door to corresponding constructions on hexahedral meshes.
		
		The analysis of the scheme combines tools from the embedded Trefftz DG abstract framework~\cite{LLSV_ARXIV_2024} and the anisotropic setting~\cite{G03dm}.
		We prove well-posedness and quasi-optimality of the method on anisotropic, possibly curved, quadrilateral elements,  and derive anisotropic a priori error estimates. 
		Numerical experiments illustrate the performance of the method under $h$- and $hp$-refinement, including curved domain examples and comparisons with the standard DG scheme.
		\vspace{-2em}
		
		\paragraph{Outline.}
		The paper is organized as follows.
		\cref{sec:DG} 
		introduces the anisotropic quadrilateral mesh setting, the discrete tensor-product spaces, the interior penalty DG formulation, and the corresponding embedded Trefftz DG method.
		\Cref{sec:stability} is devoted to the analysis of
		the method: we study the stability of the local operator on \emph{pullback} and physical elements, prove coercivity and continuity of the global bilinear form, and 
		derive well-posedness, quasi-optimality, and anisotropic a priori error estimates.
		In \cref{sect::numericalexperiments}, we present numerical experiments, including 
		$h$- and $hp$-convergence, curved-domain examples, and an anisotropic 
		diffusion problem.
		
    \subsection{Model problem}\label{sect::modelproblem}
	Let $\Omega\subset \mathbb{R}^2$ be a bounded Lipschitz domain with boundary $\partial\Omega$.
	Given a diffusivity parameter $\varepsilon>0$, a source term $f\in L^2(\Omega)$, and a Dirichlet datum $g\in H^{1/2}(\partial\Omega)$, we consider the following reaction--diffusion boundary value problem: find $u:\Omega \to \IR$ such that
	\begin{eqs}\label{eq:PDE}
		-\eps^2 \Lap u + u = f \quad &\text{in } \Omega,\\
		u = g \quad &\text{on } \partial\Omega.
	\end{eqs}
	Let $ H^1_g(\Omega):=\{v\in H^1(\Omega)\mid v_{|_{\partial\Omega}}=g\} $ denote the %
    convex set of $H^1(\Omega)$ functions satisfying the Dirichlet condition in the trace sense.
	The variational formulation of \eqref{eq:PDE} reads as follows: find $u\in H^1_g(\Omega)$ such that
		\begin{equation*}\label{eq:abstract}
		a(u,v): =	\int_\Omega \left(\eps^2 \nabla u\cdot \nabla v + uv\right)=  \int_\Omega fv
			\qquad \forall v\in H^1_0(\Omega).
		\end{equation*}
	The bilinear form $a(\cdot,\cdot)$ is coercive and continuous with respect to the energy norm
	$(\varepsilon^2 \|\nabla v\|_{L^2(\Omega)}^2 + \|v\|_{L^2(\Omega)}^2)^\frac12$. By the Lax--Milgram theorem, the variational problem~\eqref{eq:abstract} admits a unique weak solution $u\in H^1_g(\Omega)$. 

	We are particularly interested in the singularly perturbed regime $0 < \varepsilon \ll 1$, where the diffusion term is dominated by the reaction term. In this case, the solution may exhibit sharp boundary layers and strongly anisotropic features. Standard isotropic meshes may therefore become inefficient, whereas anisotropic mesh refinement can resolve such localized features more effectively with fewer degrees of freedom.
	\section{Definition of the method} \label{sec:DG}
	After introducing some standard DG notation and the mesh setting in \cref{ssec:notation}, we define the interior penalty DG formulation for the reaction--diffusion problem in \cref{ssec:DGformulation}, and then introduce the corresponding embedded Trefftz DG method in \cref{ssec:TDGformulation}.
	
	\subsection{Notation and mesh}\label{ssec:notation}
	Let $D\subset \IR^d$ ($d\in\{1,2\}$) be an open bounded Lipschitz set. We denote by $L^2(D)$ the space of Lebesgue square integrable functions on~$D$ with scalar product $(\cdot,\cdot)_{D}$ and norm $\|\cdot\|_{D}$. For $s\in\IR$, $H^s(D)$ denotes the Sobolev space of order $s$  with norm $\|\cdot\|_{H^s(D)}$.
	
	Let $\calK$ be a nonoverlapping partition of the domain $\Omega$ into \emph{curvilinear} quadrilateral elements.
	More precisely, we assume that each element~$K\in\calK$ is the image of the unit square~$\widehat{K}:= (0,1)^2$ through a map~$\GK \circ \FK$.
	The map $\FK : \Kh \to \Kt$ is an affine transformation of the form
	\begin{equation*}
	\FK( \widehat \bx) := \DK \widehat \bx+d_K, \quad \text{ for } \widehat \bx\in\widehat K, \quad \text{ with } 
	\DK=\mathrm{diag}(\hKot, \hKtt)\in \IR^{2\times 2} \quad \text{and} \quad \dK \in \IR^2,
	\end{equation*}
	where $\hKot$ and $\hKtt$ 
	denote the lengths of the edges of the \emph{pullback} tensor-product element~$\Kt = \Kt_1 \times \Kt_2$. Such edges are parallel to the~$\xt_1$-  and~$\xt_2$-axes, respectively; see Figure~\ref{fig:quadmap}.
	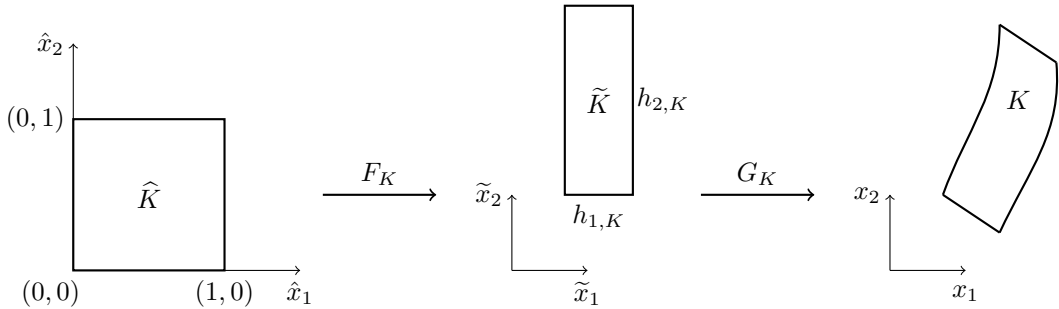
\begin{figure}[!htb]
		\begin{center}
			\begin{tikzpicture}[scale=1]
				\draw[thick] (0,0) rectangle (2,2);
				\node at (1,2.2) {};
				\node at (1,-0.2) {};
				\node[rotate=90] at (-0.4,1) {};
				\node[rotate=90] at (2.4,1) {};
							
				\draw[->] (0,0) -- (3,0);
				\draw[->] (0,0) -- (0,3);
				
				\node at (-0.3,3) {$\hat{x}_2$};
				\node at (3,-0.3) {$\hat{x}_1$};
				\node at (1,1) {$\Kh$};
				\node at (-0.3,-0.3) {$(0,0)$};
				\node at (2,-0.3) {$(1,0)$};
				\node at (-0.5,2) {$(0,1)$};
				
				\draw[->, thick] (3.3,1) -- (4.8,1) node[midway, above] {$F_K$};
				\draw[thick] (6.5,1) rectangle (7.4,3.5);
				\node at (6.95,2.25) {$\Kt$};
	            \node at (6.95,1-0.3) {$\hKot$};
	            \node at (7.8,2.25) {$\hKtt$};
	            \draw[->] (5.8,0) -- (5.8,1);
	            \draw[->] (5.8,0) -- (6.8,0);
	            \node at (5.8-0.3,1) {$\xt_2$};
	            \node at (6.8,0-0.3) {$\xt_1$};
	
				\draw[->, thick] (8.3,1) -- (9.8,1) node[midway, above] {$G_K$};
	            \draw[thick]    (11.5,1) -- (12.25,0.5);
	            \draw[thick]    (12.25,0.5) to[out=65,in=-85] (13,2.75);
	            \draw[thick]    (13,2.75) -- (12.25,3.25);
	            \draw[thick]    (12.25,3.25)  to[out=-90,in=70] (11.5,1);
	            \draw[->] (10.8,0) -- (10.8,1);
	            \draw[->] (10.8,0) -- (11.8,0);
				\node at (12.5,2.25) {$K$};
	            \node at (10.8-0.3,1) {$x_2$};
	            \node at (11.8,0-0.3) {$x_1$};
			\end{tikzpicture}
		\end{center}
        \vspace{-1em}
		\caption{Construction of the elements in~$\calK$ via composition of an affine map $F_K$ and diffeomorphism $G_K$.}
		\label{fig:quadmap}
    \vspace{-.5em}
	\end{figure}
	
	The map~$G_K: \widetilde{K} \to K$ is a~$C^1$-diffeomorphism with Jacobian~$J_{\GK} : \Kt \to \R^{2\times 2}$ satisfying 
	\begin{equation}
	\label{eq:GK}
	(C_G^{\diamond})^{-1} \le \det J_{\GK} \le C_G^{\diamond} \quad \text{ and } \quad \Norm{J_{\GK}}{L^{\infty}(\Kt)} \le C_G^* \qquad \forall K \in \calK,
	\end{equation}
	for some positive constants~$C_G^{\diamond}$ and~$C_G^*$ independent of~$K$. 

	We denote by $\partial_1$ and $\partial_2$ the derivatives with respect to the auxiliary coordinates $\xt_1$ and $\xt_2$, respectively.
	For differential operators in auxiliary coordinates, we use the tilde notation, e.g.,\ 
	\begin{equation*}
	    \widetilde{\grad} v := (\partial_1 v, \partial_2 v), 
        \qquad 
         \widetilde{\Delta} v := \partial_1^2 v + \partial_2^2 v,
        \qquad \widetilde{\Div} \bw := \partial_1 w_1 + \partial_2 w_2. 
	\end{equation*}
	
	The set of edges of $\calK$ is denoted by $\calE = \calE^I \cup \calE^D$, where~$\calE^I$ and~$\calE^D$ are the sets of interior and boundary edges, respectively.
	For~$i = 1,2$, we denote by $\calE_i^K$ the union of the two edges of $\partial K$, whose preimages under $G_K$ are parallel to the $\widetilde x_i$-axis, i.e.,\
		\begin{equation*}
			\calE_1^K:=G_K(\Kt_1 \times \partial \Kt_2), \qquad \calE_2^K:=G_K(\partial \Kt_1 \times \Kt_2).
		\end{equation*}
	Moreover, for each element $K\in\calK$ and each edge~$e\in \calE$ of~$K$, 
    given $\Kt = \Kt_1 \times \Kt_2 =G_K^{-1}(K)$, we define the adjacent edge length as follows:
	\begin{equation*}
	    \hperp^{e, \Kt} :=
		\begin{cases}
			\hKtt &  \text{if } e\in\calE^K_1,\\
	        \hKot &   \text{if } e 
	        \in\calE^K_2.
		\end{cases}
	\end{equation*}
	As stressed in our notation, differently from the length of~$e$, the adjacent edge length $\hperp^{e,\Kt}$ depends on the considered element~$\Kt$.
	
	Without loss of generality, we assume that
	\begin{equation}\label{eq:anisotropicity}
		\hKot \le \hKtt.
	\end{equation}
	This can always be achieved by interchanging the roles of $\hKot$ and $\hKtt$, which corresponds to swapping the auxiliary coordinates $\widetilde x_1$ and $\widetilde x_2$ in the definition of the pullback element $\Kt$. This is an orthonormal transformation and therefore does not affect the properties of the mapping $G_K$. 

	For each~$e \in\calE^I$, 
	we set~$\bn_e$ as one of the two unit normal vectors orthogonal to~$e$, and we use the convention that~$\bn_e$ points from~$K^-$ to~$K^+$, where~$K^-$ and~$K^+$ are the neighboring elements in~$\calK$ sharing the edge~$e$.
	In what follows, we use the notation~$\hperp^{e, -}$ and~$\hperp^{e, +}$ to indicate the corresponding adjacent lengths of~$e$.
	Moreover, we 
	shall omit the superscript~$e$ whenever the dependence of the adjacent length
	on~$e$ is clear.
	
	Finally, we define the following weighted average~$(\mvl{\cdot}_{\hperp^e})$ and 
	jump~$(\jmp{\cdot})$ operators for piecewise smooth scalar~$(v)$ and vector-valued~$(\bw)$ functions:
	\begin{align}\label{eq:averages-jumps}
	     \mvl{\bw}_{\hperp^e} &:= \Big(\frac{\hperp^-}{\hperp^+ + \hperp^-}\Big) \bw^- + \Big(\frac{\hperp^+}{\hperp^+ + \hperp^-}\Big) \bw^+ , 
	    \quad \jmp{v} := (v^- - v^+)\bn_e, 
	\end{align}
	where 
	$v^+$ and~$v^-$ denote the traces of~$v$ on~$e$ taken from the interior of $K^+$ and~$K^-$, respectively.
	For boundary edges $e \in \calE^D$ such that $e\subset \partial K$ for some $K\in\calK$, we define the average and jump operators by
	\[
	\mvl{\bw}_{\hperp^e} := \bw|_K, 
	\qquad 
	\jmp{v} := v|_K  \bn_{\Omega},
	\]
	where $\bn_{\Omega}$ denotes the outward unit normal to $\partial\Omega$.
	
	Given a degree of approximation~$p \in \N$ with~$p \geq 1$, the corresponding space of tensor-product polynomials on the pullback element~$\Kt = \Kt_1 \times \Kt_2$ is defined as
	\begin{equation*}
		\IP^p_\otimes(\Kt):= \IP^p(\Kt_1) \otimes \IP^p(\Kt_2),
	\end{equation*}
	where~$\otimes$ denotes the algebraic
	product of vector spaces, and~$\IP^p(I)$ denotes the space of polynomials defined on the interval~$I$ of degree less than or equal to~$p$.
	
	On the physical element~$K \in \calK$, we use the following \emph{pushforward} space: 
	\begin{equation*}
	\VK
	 := \{v \in L^2(K) \mid (v \circ \GK) \in \IP^p_\otimes(\Kt)\}.
	\end{equation*}
	In other words, functions in~$\VK$ are tensor-product polynomials in the auxiliary coordinates $(\xt_1, \xt_2)$, but not necessarily in the physical coordinates~$(x_1, x_2)$. The global discrete space is defined as 
	\begin{equation*}
		\Vh:= \prod_{K \in \calK} \VK.
	\end{equation*}
	Here, and throughout the article, the subscript $h$ is used to denote the discrete quantities for simplicity. The analysis  is carried out in terms of the elemental mesh sizes $\hKot$ and $\hKtt$.
	
	\subsection{DG formulation}\label{ssec:DGformulation}
	Let $\Vsh:=\Vh+H^2(\Omega)$.
	We consider the following interior penalty DG formulation: 
	\begin{equation*}
		\text{Find } \uh \in \Vh \quad \text {s.t.} \quad 
		\ah(\uh, \vh) = \ell_h(\vh) \quad \forall \vh \in \Vh,
	\end{equation*}
	where the bilinear form~$\ah : \Vsh \times \Vh \to \R$ and the linear functional~$\lh : \Vh \to \R$ are given by
	\begin{align}\label{eq:ah}
	\nonumber
		\ah(w, \vh) & := \sum_{K\in \calK} \int_K (\varepsilon^2 \nabla w \cdot \nabla \vh +  w \vh)
		 \\
	\nonumber
	    & \quad +
	\sum_{e \in\calE^{\calI}} \int_{e} (-\mvl{\varepsilon^2 \nabla w}_{\hperp^e} \cdot \jmp{\vh} - \mvl{\varepsilon^2 \nabla \vh}_{\hperp^e} \cdot \jmp{w} +  \sigma\frac{\eps^2}{h^{+}_\perp+h^{-}_\perp} \jmp{w}\cdot\jmp{\vh}) 
	    \\
	    & \quad + \sum_{e\in \calE^{\calD}} \int_{e} (-\varepsilon^2 \nabla w \cdot \bn_{\Omega} \vh - \varepsilon^2 \nabla \vh \cdot \bn_{\Omega} w + \sigma\frac{\eps^2}{\hperp} w \vh) 
	, \\
	    \label{eq:lh}
		\lh(\vh) & := \sum_{K\in \calK} \int_K f \vh 
	 +\sum_{e\in \calE^{\calD}} \int_{e} g (-\varepsilon^2 \nabla \vh \cdot \bn_{\Omega} + \sigma \frac{\eps^2}{\hperp} \vh),
	\end{align}
    with~$\sigma>0$ a ``sufficiently large" dimensionless penalty parameter.
	
	\begin{remark}[Stability term and weighted averages]
	Our choice of the stability term differs from the one used in~\cite[Lemma 8.1]{G06hp}. In that work, the~$\hperp$-dependent factor~$(1/\hperp^+ + 1/\hperp^-)$ is used, which results in a strong dependence of the local contribution from a single element 
	to the energy norm on the adjacent edge lengths of both neighboring elements. That strong dependence would complicate the stability analysis of the proposed embedded Trefftz DG method,  which relies on the study of such local contributions.
	
	On the other hand, the~$\hperp$-dependent factor we employ can be easily bounded as $$\frac{1}{\hperp^+ + \hperp^-} \le \frac{1}{\hperp^{\pm}},$$ 
	thus avoiding the issue mentioned above. Accordingly, the weighted-average parameters in~\eqref{eq:averages-jumps} are defined to guarantee the stability of the global bilinear form~$\ah(\cdot, \cdot)$.
	\end{remark}

	\subsection{Embedded Trefftz DG formulation}\label{ssec:TDGformulation}
	The embedded Trefftz DG method is obtained by restricting the
	discrete solution to a subspace of~$\Vh$ whose elements satisfy,
	in a weak sense,
    equation~\eqref{eq:PDE} elementwise.
	The construction of the embedded Trefftz space relies on two key ingredients on each element $K \in \calK$: a local differential operator $A_K$ and a suitably designed test space $\IL_h(K)$.

	Let $K\in \calK$.
	On the pullback element $\Kt = \Kt_1 \times \Kt_2=G_K^{-1}(K)$, we define
	\begin{equation}\label{def:Lspace}
		\IL_h(\Kt) := \{\tilde v_h \in \IP^p_\otimes(\Kt) \mid \tilde v_h|_{\partial \Kt_1 \times \Kt_2} = 0\}.
	\end{equation}
This space consists of tensor-product polynomials on $\Kt$ that vanish on the two longer edges, i.e.,\ the edges parallel to the $\widetilde x_2$-axis.
We define the local test space on the physical element $K$ as the pushforward of $\IL_h(\Kt)$:
	\begin{equation*}
		\LKh := \{v_h \in \Vh(K) \mid (v_h\circ G_K) \in \IL_h(\Kt)\}.
	\end{equation*}
    \begin{remark}[Design of the test space]\label{rem:testspace}
        The choice of the test space $\IL_h(\Kt)$ in \eqref{def:Lspace} is motivated by the fact that functions in this space satisfy a Poincaré--Steklov inequality which involves only the shorter edge length of the element, see \eqref{eq:SteklovPoincare} below. 
        This is crucial for the	
        local stability 
        analysis of \cref{sec:localstability_pullback}. 
        In two dimensions the $H^1$-norm along the longer edge can be bounded by the $H^1$-seminorm along the shorter edge, in the following sense 
        $\|\partial_2 \tilde u\|_{\Kt} \lesssim \frac{\hKot}{\hKtt} \|\partial_1 \tilde u\|_{\Kt}$, for all $\tilde u \in \IL_h(\Kt).$
        Details are given in \cref{lem:anisotropic-inverse}.
        We want to highlight, that in higher dimensions, this remains true, with the $H^1$-seminorm involving only the derivatives along the shorter edge.
	\end{remark}

Denoting by $\LKh'$ 
the dual space of $\LKh$, we define the local reaction--diffusion operator $\AK : \Vsh|_K \to \LKh'$ as
	\begin{equation*}
		A_K v := -\eps^2 \Lap v +v \qquad \forall v \in \Vsh|_K .
	\end{equation*}
	Then the local embedded Trefftz space is defined by
	\begin{equation*} 
	    \ITh(K) : = \{v_h \in \Vh(K) \mid (\AK v_h, q_h)_{K}=0 \quad \forall q_h\in \LKh\},
	\end{equation*}
	and the global space is given by $\ITh:= \prod_{K \in \calK} \ITh(K)$.
		A straightforward dimension count gives
		\begin{equation*}
			\dim\big(\IT_h(K)\big)=\dim\big(\IP^p_{\otimes}(K)\big)-\dim\big(\IL_h(K)\big)=(p+1)^2-(p+1)(p-1)=2p+2.
		\end{equation*}
		Note that the dimension of the embedded Trefftz space grows only linearly 
		in $p$, in contrast to the $(p+1)^2$ degrees of freedom of the full 
		tensor-product space $\IP^p_\otimes(K)$. This leads to a significant reduction of the globally coupled 
		degrees of freedom, see Table~\ref{tab:ndof-nnze}.
	
	The embedded Trefftz DG method reads
	\begin{eqs}\label{eq:PDEh}
	    \text{Find } u_h \in \Vh
	    \text{ such that } &&\quad\inner{\AK  u_h, v_\IL}_{K}&=\ell_K(v_\IL) &&\quad \forall v_\IL\in \IL_h(K), \ \forall K\in\calK,
	 \\ \text{ and }  && \quad a_h(u_h,v_\IT)&=\ell_h(v_\IT) &&\quad \forall v_\IT \in \ITh,
	\end{eqs}
	where $\ell_K(v):=(f,v)_{K}$, and $a_h(\cdot, \cdot)$ and $\ell_h(\cdot)$ are defined in \eqref{eq:ah} and \eqref{eq:lh}, respectively.
	
	\paragraph{On the implementation.}
	The discrete formulation \eqref{eq:PDEh} is used for the analysis but not in the implementation.
	In practice, the Trefftz space $\ITh$ is never explicitly constructed. 
	Instead, we build an embedding based on the singular value decomposition (SVD) of the small matrices representing the local operators $A_K$. We explicitly construct the local test space $\LKh$ and its pullback $\IL_h(\Kt)$, as they are required to compute the SVD, and thus to construct the embedding.
	We stress that the global degrees of freedom that remain after the embedding are only those associated with 
    the Trefftz space, leading to a significant reduction of the computational cost of the method.
	Further implementation details are provided in~\cref{sect::numericalexperiments}, and we also refer to~\cite{LS_IJMNE_2023}.
	
	\section{Analysis of the embedded Trefftz DG method}\label{sec:stability}
	The analysis proceeds in several steps.
	In \cref{sec:local}, we study stability and continuity of the local operator $A_K$,  and  in~\cref{sec:global}, we 
	prove coercivity and continuity of the global bilinear form 
    $a_h(\cdot, \cdot)$.
	Following the framework of~\cite{LLSV_ARXIV_2024}, these results imply well-posedness and quasi-optimality of the embedded Trefftz DG method \eqref{eq:PDEh}; the corresponding statements are presented in \cref{sec:wp-qo}. 
	Finally, in \cref{sec:apriori}, we derive anisotropic a priori error estimates.
	
	In the rest of the paper, we use the notation $\lesssim$ for inequalities where we 
    omit
    generic constants that are
    independent of $\eps$, $\hKot$, $\hKtt$ for all $K\in\calK$, and of any functions involved in the estimate.
	
	For all $v\in \Vsh=\Vh+H^2(\Omega)$, we define the following  mesh-dependent norms:
	\begin{equation}\label{eq:Vhnorm}
		\begin{split}
	    &\norm{v}_{\Vh}^2:=\sum_{K\in\calK}\left(\varepsilon^2\norm{\nabla v}^2_{K}+\norm{v}^2_{K}\right)+\abs{v}_{\mathrm{J}}^2, \quad
		\abs{v}_{\mathrm{J}}^2:=\sum_{e\in\calE^{\calI}}\int_{e} \frac{\eps^2}{h^{+}_\perp+h^{-}_\perp} \jmp{v}^2
	    +\sum_{e\in\calE^{\calD}}\int_{e} \frac{\eps^2}{\hperp} v^2,\\
	    &\dgsnorm{v}^2:=
		\norm{v}_{\Vh}^2 + 
		\sum_{K\in\calK}\sum_{e\subset\partial  K} \hperp^{e, \Kt} \norm{\eps \nabla  v\cdot \bn}^2_{e}.
	\end{split}
	\end{equation}
For each $K\in\calK$, we define a norm 
in the pullback test space 
$\IL_h(\Kt)$ by
\begin{equation}\label{eq:norm_Lh}
	\|\tilde v\|^2_{\mathbb{L}_h(\Kt)}
	:= \norm{\tilde v}^2_{\Kt}
	+\eps^2\norm{\partial_1\tilde v}^2_{\Kt}
	\qquad \forall \tilde v\in \IL_h(\Kt),
\end{equation}
and equip the test space $\LKh$ with the norm
$\|v\|_{\LKh}:=\|v\circ G_K\|_{\mathbb{L}_h(\Kt)}$.
The dual space $\LKh'$ is endowed with the norm $\|\cdot\|_{\IL_h(K)'}:= \sup_{v\in \IL_h(K)\setminus \{0\}}\frac{\langle\cdot, v\rangle}{\|\cdot\|_{\IL_h(K)}}$, where $\langle \cdot, \cdot \rangle$ denotes the duality pairing between $\IL_h(K)'$ and $\IL_h(K)$.

We follow 
the abstract framework of embedded Trefftz DG methods developed in~\cite{LLSV_ARXIV_2024}, and adapt it to the present setting. 
	Well-posedness of problem~\eqref{eq:PDEh} holds under the following local and global properties:

\paragraph{Local properties.} There exist constants $\cA,\CA>0$, independent of $\hKot$ and $\hKtt$ for all $K\in\calK$, such that
	\begin{align}
		\norm{\AK v_h}_{\LKh'} &\ge \cA \norm{v_h}_{\Vh}
		&&\forall v_h \in \Lh(K),\ \forall K\in\calK, \label{eq:A_coerc}\\
		\sum_{K\in\calK} \norm{\AK v}_{\LKh'}^2 &\le \CA^2 \dgsnorm{v}^2
		&&\forall v\in \Vsh. \label{eq:A_cont}
	\end{align}

	\paragraph{Global properties.} There exist constants $\ca,\Ca>0$, independent of  $\hKot$ and $\hKtt$ for all $K\in\calK$,  such that
	\begin{align}
		a_h(v_h,v_h) &\ge \ca \norm{v_h}_{\Vh}^2
		&&\forall v_h\in \ITh, \label{eq:acoerc}\\
		a_h(v,v_h) &\le \Ca \dgsnorm{v}\,\norm{v_h}_{\Vh}
		&&\forall v\in \Vsh,\ \forall v_h\in \ITh. \label{eq:ah_cont}
	\end{align}
	
We recall the following anisotropic trace inequality from~\cite[Corollary 3.49]{G03dm}, which will be used in the analysis.
	\begin{lemma}[Anisotropic trace inequality]
		Let $K\in\calK$ be a mesh element and let $\Kt = \Kt_1 \times \Kt_2=G_K^{-1}(K)$ be its pullback. Then
		\begin{equation}\label{eq:anisotropictraceineq}
			\norm{\tilde v}_{\Kt_1\times\partial\Kt_2}^2\lesssim \hKtt^{-1}\norm{\tilde v}_{\Kt}^2 
			\ \  \text{ and } \ \
			\norm{\tilde v}_{\partial\Kt_1\times\Kt_2}^2\lesssim \hKot^{-1}\norm{\tilde v}_{\Kt}^2 
			\quad \forall \tilde v \in \IP^p_\otimes(\Kt).
		\end{equation}
	\end{lemma}

	\subsection{Local problem}\label{sec:local}
	To establish the local stability \eqref{eq:A_coerc} of the  operator $A_K $ on $\IL_h(K)$, we first prove in \cref{sec:localstability_pullback} that \eqref{eq:A_coerc} 
    holds for a prototype operator on $\IL_h(\Kt)$.
    Via the perturbation argument of Lemma \ref{lem:abstractneumann}, we then show in~\cref{sec:localstability_physical} that a pulled-back operator  inherits 
	this stability, which  then yields 
	stability of $A_K$ on the physical element.
	Finally, continuity of $A_K$ is established in \cref{sec:localcontinuity}.

	Before proceeding, we state two preliminary results that are used 
	throughout the analysis.

	The specific construction of the test space $\IL_h(\Kt)$ allows us to prove the following anisotropic inverse inequality, which is crucial for the stability analysis of the method.
	\begin{lemma}[Anisotropic inverse inequality for $\IL_h(\Kt)$]\label{lem:anisotropic-inverse}
	Let $K\in\calK$ be a mesh element and let $\Kt = \Kt_1 \times \Kt_2=G_K^{-1}(K)$ be its pullback. 
	Then, there exists a constant $\const{ainv}>0$ independent of $\hKot$ and $\hKtt$ such that 
	\begin{equation}\label{ainv}
	    \|\partial_2 \tilde v\|_{\Kt} \leq \const{ainv} \frac{\hKot}{\hKtt} \|\partial_1 \tilde v\|_{\Kt} \quad \forall \tilde v \in \IL_h(\Kt).
	\end{equation}
	\end{lemma}
	\begin{proof}
	Since $\tilde v\in \mathbb L_h(\Kt)$ vanishes
	on $\partial\Kt_1\times \Kt_2$,
	the Poincaré--Steklov inequality on the interval
	$\Kt_1$ yields
	\begin{equation}\label{eq:SteklovPoincare}
	\|\partial_1 \tilde v\|_{\Kt}^2 = \int_{\Kt_2} \|\partial_1 \tilde v\|_{\Kt_1}^2 \,\mathrm{d}\tilde x_2
	\gtrsim \hKot^{-2} \|\tilde v\|_{\Kt}^2.
	\end{equation}
	Moreover, by standard polynomial inverse estimates, we have
	\begin{equation*}
	\|\partial_2 \tilde v\|_{\Kt}^2 
	= \int_{\Kt_1} \|\partial_2 \tilde v\|_{\Kt_2}^2 \,\mathrm{d}\tilde x_1
	\lesssim \hKtt^{-2} \|\tilde v\|_{\Kt}^2.
	\end{equation*}
	Combining these two inequalities gives the result.
	\end{proof}
	The following 
    lemma strengthens the Cauchy--Schwarz inequality for the second derivative of polynomials.
	The constant $C_p$ in~\eqref{eq:noeigenpoly} below depends adversely on the polynomial degree $p$, and it is expected to approach~$0$ as $p\to\infty$.
	However, for any fixed $p$,
    $C_p$ is strictly larger than $0$.
	\begin{lemma}\label{lem:noeigenpoly}
	Given a polynomial degree $p\in\mathbb{N}$, 
    we have 
	\begin{equation}\label{eq:noeigenpoly}
		|(\phi'', \phi)_{[0,1]} | \leq (1-C_p)\,\norm{\phi''}_{[0,1]} \,\norm{\phi}_{[0,1]} \qquad \forall \phi \in \IP^p([0, 1]), 
		\end{equation}
		for some constant $0<C_p<1$ depending only on $p$. 
		\end{lemma}
		\begin{proof}
		By the Cauchy--Schwarz inequality,
		\begin{equation}
        \label{eq:equality-inequality}
		|( \phi'', \phi)_{[0,1]} | \;\le\; \norm{\phi''}_{[0,1]} \,\norm{\phi}_{[0,1]} \qquad \forall \phi \in \IP^p([0,1]).
		\end{equation}
		Equality holds if and only if~$\phi$ satisfies $\phi''=\lambda \phi$ for some $\lambda\in\IR$, i.e.,\ if $\phi$ is an eigenfunction of the differential operator $u\mapsto u''$. However, the only 
        eigenfunctions of such an operator are trigonometric or exponential functions (or affine functions for $\lambda=0$). 
		Apart from the trivial affine case, none of these are polynomials. 
        Hence, on the compact set
		\begin{equation*}
		S := \{\phi \in \IP^p([0,1]) \mid \norm{\phi''}_{[0,1]} = 1\} \subset \IP^p([0,1]),
		\end{equation*}
		no $\phi\in S$ achieves equality in~\eqref{eq:equality-inequality}.  
		Now, define the following maximum over $S$:
		\begin{equation*}
		s_p := \max_{\phi \in S}\frac{ |(  \phi'', \phi)_{[0,1]} |}{  \norm{\phi}_{[0,1]}} < 1.
		\end{equation*}
		Setting $C_p := 1-s_p$ yields~\eqref{eq:noeigenpoly}.
		\end{proof}

\subsubsection{Local stability on the pullback element}\label{sec:localstability_pullback}
Given $K\in\calK$, we introduce a prototype operator $\AKO$ on the pullback element $\Kt=G_K^{-1}(K)$, defined as
	\begin{equation}\label{eq:AK0}
	    \AKO  \tilde v := -\eps^2 
        \widetilde{\Delta}\tilde  v +   \tilde v \qquad 
        \forall \tilde v \in 
	    \IP^p_\otimes(\tilde K).
	\end{equation}
Recall that the norm on $\mathbb{L}_h(\Kt)$ is given by \eqref{eq:norm_Lh}, namely
$
\|\tilde v\|^2_{\mathbb{L}_h(\Kt)}
:= \norm{\tilde v}^2_{\Kt}
+\eps^2\norm{\partial_1\tilde v}^2_{\Kt}
$.
In the following lemma, we prove that $A_{\Kt,0}$ satisfies the stability estimate \eqref{eq:A_coerc} on $\IL_h(\Kt)$.
In fact, we show a stronger result, where the norm on the right-hand side exhibits the scaling of the operator with respect to the anisotropic mesh sizes.
	\begin{lemma}[Stability estimate for~$A_{\Kt, 0}$]\label{lem:localcoercivity}
		There exists a constant $c_{A,0}>0$ such that
	\begin{equation*}
		\|\AKO  \tilde v \|_{\IL_h(\Kt)'} 
	    \geq c_{A,0}
	    \Big( \norm{\tilde v}_{\Kt}^2 + \eps^2 \norm{\partial_1 \tilde v}_{\Kt}^2 + \Big(\const{ainv}\frac{\hKot}{\hKtt}\Big)^{-2} \eps^2 \norm{\partial_2 \tilde v}_{\Kt}^2 \Big)^{\frac{1}{2}}
	    \quad \forall \tilde v\in \IL_h(\Kt),
	\end{equation*}
    and, in particular,
    it holds 
	\begin{equation*}
		\|\AKO  \tilde v \|_{\IL_h(\Kt)'} 
        \geq c_{A,0} \norm{\tilde v}_{\mathbb{L}_h(\Kt)}
	    \quad \forall \tilde v\in \IL_h(\Kt).
	\end{equation*}
	\end{lemma}
	\begin{proof}
	For \(p=1\), the space \(\IL_h(\Kt)\) is trivial, so there is nothing to prove. Hence assume \(p\ge2\).
	
	Let \(\{\psi_i\}_{i=1}^{p-1}\subset \IP^p(\Kt_1)\cap H_0^1(\Kt_1)\) be a basis that is orthonormal in $L^2(\Kt_1)$ and orthogonal with respect to the inner product
$
	(\cdot,\cdot)_{\Kt_1}
	+\eps^2(\partial_1\cdot,\partial_1\cdot)_{\Kt_1}
$,  i.e., such that
	\begin{equation}\label{eq:d1ortho}
	(\psi_i,\psi_j)_{\Kt_1}
	+\eps^2(\partial_1\psi_i,\partial_1\psi_j)_{\Kt_1}
	=\lambda_i\delta_{ij},
	\qquad i,j=1,\dots,p-1.
	\end{equation}
	In particular, for \(i\neq j\),
	$
	(\psi_i,\psi_j)_{\Kt_1}=0$, and
	$(\partial_1\psi_i,\partial_1\psi_j)_{\Kt_1}=0$, and there exist $\widehat\lambda_i>0$, depending only on $p$, such that
	\begin{equation*}
    \lambda_i=1+\eps^2\hKot^{-2}\widehat\lambda_i.
	\end{equation*}
	
	Let \(\tilde v\in \IL_h(\Kt)\) and write
	\begin{equation*}
	\tilde v(\widetilde x_1,\widetilde x_2)=\sum_{i=1}^{p-1}\psi_i(\widetilde x_1)\tilde v_i(\widetilde x_2),
	\quad \text{for }\tilde v_i\in \IP^p(\Kt_2),  i = 1, \ldots, p-1.
	\end{equation*}
	
	The functions $\{\psi_j  \tilde v_j\}_{j=1}^{p-1}$ 
    are orthogonal on $\Kt$ with respect to several operators. More precisely, for all $i \neq j$, we have
	\begin{equation}\label{eq:psiortho}
	\begin{cases}
	( \psi_i  \tilde v_i, \psi_j  \tilde v_j )_{\Kt} = 0,\\
	(\psi_i  \partial_2^2 \tilde v_i, \psi_j  \tilde v_j )_{\Kt} = 0,\\
	( \partial_1^2 \psi_i  \tilde v_i, \psi_j  \tilde v_j )_{\Kt} = 0.
	\end{cases}
	\end{equation}
	The first two equalities follow from the orthogonality of $\{\psi_i\}_{i=1}^{p-1}$ in $L^2(\Kt_1)$ and the tensor-product structure of~$\Kt$. For the last one, using integration by parts in the $\widetilde x_1$-variable and the fact that $\psi_j=0$ on $\partial\Kt_1$, we obtain
	\begin{equation*}
	(\partial_1^2\psi_i\,\tilde v_i,\psi_j\,\tilde v_j)_{\Kt}
	=
	-(\partial_1\psi_i,\partial_1\psi_j)_{\Kt_1}
	(\tilde v_i,\tilde v_j)_{\Kt_2}
	=0,
	\end{equation*}
	for \(i\neq j\), where we used \eqref{eq:d1ortho}.
	Using \eqref{eq:psiortho} and \eqref{eq:d1ortho}, we obtain the following expression for the norm of $\tilde v$:
	\begin{equation}\label{eq:norm-psiortho}
		\begin{split}
			\|\tilde v\|_{\mathbb L_h(\Kt)}^2
			=
			\|\tilde v\|_{\Kt}^2
			+\eps^2\|\partial_1\tilde v\|_{\Kt}^2
			&=
			\sum_{i=1}^{p-1}
			\left(
			\|\tilde v_i\|_{\Kt_2}^2
			+\eps^2\|\partial_1\psi_i\|_{\Kt_1}^2
			\|\tilde v_i\|_{\Kt_2}^2
			\right)
			\\
			&=
			\sum_{i=1}^{p-1}\lambda_i\|\tilde v_i\|_{\Kt_2}^2
			=
			\sum_{i=1}^{p-1}\lambda_i\|\psi_i\tilde v_i\|_{\Kt}^2 .
		\end{split}
	\end{equation}
	Now let \(\tilde u\in \IL_h(\Kt)\) and write 
	\begin{equation*}
	\tilde u(\widetilde x_1,\widetilde x_2)
	=
	\sum_{i=1}^{p-1}\psi_i(\widetilde x_1)\tilde u_i(\widetilde x_2),
	\quad  \text{for } \tilde{u}_i\in \IP^p(\Kt_2), \ i = 1, \ldots, p-1.
	\end{equation*}
	For any $i\in\{1,\dots,p-1\}$, set
	\begin{equation*}
	u_i:=\psi_i\tilde u_i,
	\qquad
	v_i:=u_i-\frac{\eps^2}{\lambda_i}\partial_2^2u_i
	=
	\psi_i\Big(\tilde u_i-\frac{\eps^2}{\lambda_i}\partial_2^2\tilde u_i\Big)\in \IL_h(\Kt).
	\end{equation*}
    
	By \eqref{eq:d1ortho}, these choices satisfy
	$ (u_i-\eps^2\partial_1^2u_i,v_i)_{\Kt} = \lambda_i(u_i,v_i)_{\Kt}$.
	Therefore,
	\begin{align*}
	(\AKO u_i,v_i)_{\Kt}
	&=
	(u_i-\eps^2\partial_1^2u_i-\eps^2\partial_2^2u_i,v_i)_{\Kt}
	\\
	&=
	\lambda_i(u_i,v_i)_{\Kt}
	-(\eps^2\partial_2^2u_i,v_i)_{\Kt}
	\\
	&=
	\lambda_i\|u_i\|_{\Kt}^2
	-2(u_i,\eps^2\partial_2^2u_i)_{\Kt}
	+\frac{1}{\lambda_i}\|\eps^2\partial_2^2u_i\|_{\Kt}^2,
	\end{align*}
	where we used integration by parts, the fact that $\psi_i$ vanishes on the boundary $\partial \Kt_1$, and the 
    orthogonality relations~\eqref{eq:psiortho}.
	
	Using Lemma~\ref{lem:noeigenpoly} in the \(\widetilde x_2\)-variable and the weighted Young inequality, we get
	\begin{align*}
	2\big|(u_i,\eps^2\partial_2^2u_i)_{\Kt}\big|
	&\le
	2(1-C_p)\|u_i\|_{\Kt}
	\|\eps^2\partial_2^2u_i\|_{\Kt}
	\\
	&\le
	(1-C_p)
	\Big(
	\lambda_i\|u_i\|_{\Kt}^2
	+\frac1{\lambda_i}\|\eps^2\partial_2^2u_i\|_{\Kt}^2
	\Big),
	\end{align*}
	and we infer
	\begin{equation}\label{eq:mode-coercive}
	(\AKO u_i,v_i)_{\Kt}
	\ge
	C_p
	\Big(
	\lambda_i\|u_i\|_{\Kt}^2
	+\frac1{\lambda_i}\|\eps^2\partial_2^2u_i\|_{\Kt}^2
	\Big).
	\end{equation}
	Moreover, since \eqref{eq:d1ortho} also gives $ \|v_i\|_{\mathbb L_h(\Kt)}^2 = \lambda_i\|v_i\|_{\Kt}^2$,
	 then by \eqref{eq:mode-coercive} and the triangle inequality, 
	\begin{align*}
	\frac{(\AKO u_i,v_i)_{\Kt}}{\|v_i\|_{\mathbb L_h(\Kt)}}
	&\ge
	C_p\,
	\frac{
	\lambda_i\|u_i\|_{\Kt}^2
	+\lambda_i^{-1}\|\eps^2\partial_2^2u_i\|_{\Kt}^2
	}{
	\lambda_i^{1/2}\|v_i\|_{\Kt}
	}
	\\
	&\ge
	C_p\,
	\frac{
	\lambda_i\|u_i\|_{\Kt}^2
	+\lambda_i^{-1}\|\eps^2\partial_2^2u_i\|_{\Kt}^2
	}{
	\lambda_i^{1/2}\|u_i\|_{\Kt}
	+\lambda_i^{-1/2}\|\eps^2\partial_2^2u_i\|_{\Kt}
	}
	\ge
	\frac{C_p}{2}\,\lambda_i^{1/2}\|u_i\|_{\Kt}.
	\end{align*}
	For every $i=1,\dots,p-1$, using again the orthogonality relations \eqref{eq:psiortho}, we have
	$
	(\AKO \tilde u,v_i)_{\Kt}
	=
	(\AKO u_i,v_i)_{\Kt}
	$
so that
	\begin{equation*}
	\|\AKO\tilde u\|_{\mathbb L_h(\Kt)'}
	=
	\sup_{q\in \mathbb L_h(\Kt)}
	\frac{(\AKO\tilde u,q)_{\Kt}}{\|q\|_{\mathbb L_h(\Kt)}}
	\ge
	\frac{(\AKO\tilde u,v_i)_{\Kt}}{\|v_i\|_{\mathbb L_h(\Kt)}}
	\ge
	\frac{C_p}{2}\lambda_i^{1/2}\|u_i\|_{\Kt} .
	\end{equation*}
	Since this holds for every \(i\), taking the maximum over \(i\) and using
	\(
	\max_{1\le i\le p-1} a_i \ge \frac1{p-1}\sum_{i=1}^{p-1} a_i
	\)
	for nonnegative numbers \(a_i\), we get 
	\begin{align*}
	\|\AKO\tilde u\|_{\mathbb L_h(\Kt)'}^2
	&\ge
	\frac{C_p^2}{4}
	\max_{1\le i\le p-1}\lambda_i\|u_i\|_{\Kt}^2
	\ge
	\frac{C_p^2}{4(p-1)} \sum_{i=1}^{p-1}\lambda_i\|u_i\|_{\Kt}^2.
	\end{align*}
	Using the anisotropic inverse estimate \eqref{ainv} and 
    identity~\eqref{eq:norm-psiortho}, we finally obtain
	\begin{align*}
	\|\AKO\tilde u\|_{\mathbb L_h(\Kt)'}^2
	&\gtrsim
	\|\tilde u\|_{\Kt}^2
	+\eps^2\|\partial_1\tilde u\|_{\Kt}^2
	+\Big(\const{ainv}\frac{\hKot}{\hKtt}\Big)^{-2}\eps^2\|\partial_2\tilde u\|_{\Kt}^2
	.
	\end{align*}
	This proves the claim.
	\end{proof}

	\subsubsection{Local stability on the physical element} \label{sec:localstability_physical}
	In this section, we transfer 
    the stability of~$A_{\Kt,0}$ to a pulled-back operator via the perturbation argument of Lemma \ref{lem:abstractneumann}, and then to the physical operator $A_K$.
	
	Given the $C^1$ diffeomorphism $G_K: \widetilde{K} \to K$ with Jacobian~$J_{\GK} : \Kt \to \R^{2\times 2}$ satisfying~\eqref{eq:GK}, we set
	\begin{equation*}
	\rho_{G_K}(\tilde \bx):= \det(J_{\GK}(\tilde \bx)), \qquad \nu_{G_K}(\tilde \bx):=
    \rho_{\GK}(\bx) J_{\GK}(\tilde \bx)^{-1}J_{\GK}(\tilde \bx)^{-\top},
	 \qquad \text{for } \tilde \bx \in \Kt.
	\end{equation*}
	For $v\in \LKh$, we write $\tilde v:=v\circ G_K\in \mathbb L_h(\Kt)$ and define the pulled-back operator
	\begin{equation}\label{eq:AKtilde}
	A_{\Kt} \tilde v:= -\eps^2 \widetilde\Div(\nu_{G_K} \widetilde\nabla \tilde v) + \rho _{G_K}\tilde v.  
	 \end{equation}
	 The operator $A_{\Kt}$ is the pullback of $A_K$
	 to $\Kt$, in the sense that
	 \begin{equation*}
	 	(A_K v,\,q)_{K}
	 	=(A_{\Kt}\tilde v,\,\tilde q)_{\Kt}
	 	\qquad\forall v, q\in\LKh,  \quad \tilde v=v\circ G_K, \quad  \tilde q=q\circ G_K.
	 \end{equation*}
	 Moreover, $A_{\Kt}$ reduces to the prototype $A_{\Kt,0}$ 
	 when $G_K$ is a rigid motion, i.e.,\ when $\nu_{G_K}=I$ and $\rho_{G_K}=1$.
	Therefore, $A_{\Kt}$ can be viewed as a 
	perturbation of $A_{\Kt,0}$, with the perturbation controlled 
	by the deviation of $G_K$ from a rigid motion.
	To quantify this deviation, we write
	\begin{equation*}
		\nu_{G_K} - I =
		\begin{pmatrix}
			\eta_{11} & \eta_{12}\\
			\eta_{21} & \eta_{22}
		\end{pmatrix},
	\end{equation*}
	and define the  quantity
	\begin{equation}\label{eq:thetaK-aniso}
		\theta_K :=
		\|\rho_{G_K}-1\|_{L^\infty(\Kt)}
		+\|\eta_{11}\|_{L^\infty(\Kt)}
		+\frac{\hKot}{\hKtt}\bigl(\|\eta_{12}\|_{L^\infty(\Kt)}
		+\|\eta_{21}\|_{L^\infty(\Kt)}\bigr)
		+\left(\frac{\hKot}{\hKtt}\right)^2\|\eta_{22}\|_{L^\infty(\Kt)}.
	\end{equation}
    Note that any relevant scaling of the element is done in the mapping $F_K$; therefore, in the definition of $\theta_K$, we compare $\nu_{G_K}$ to the identity matrix $I$ and $\rho_{G_K}$ to $1$. 

	\begin{lemma}\label{lem:perturb-bent-aniso}
	Let $K\in\calK$ and let $\Kt=\Kt_1\times \Kt_2=G_K^{-1}(K)$.
	Let also $\AKO$ and $\widetilde A_K$ be the prototype and pulled-back operators from~\eqref{eq:AK0} and~\eqref{eq:AKtilde}, respectively. 
	Assume that  $\theta_K$ defined in~\eqref{eq:thetaK-aniso} satisfies
	$ \theta_K<\frac{c_{A,0}}{C_{\rm pert}}$, 	
		with  $c_{A,0}$ as in Lemma~\ref{lem:localcoercivity} and $C_{\rm pert}>0$ 
        as in~\eqref{eq:perturb-aniso-bound} below.
	Then
	\begin{equation*}
	\|\widetilde A_K \tilde v\|_{\mathbb L_h(\Kt)'}
	\ge
	c_{A,0}(1-\gamma_K)\|\tilde v\|_{\mathbb L_h(\Kt)}
	\qquad\forall \tilde v\in\mathbb L_h(\Kt),
	\end{equation*}
	with
	$ \gamma_K:=\frac{C_{\rm pert}\,\theta_K}{c_{A,0}}<1.$
	\end{lemma}
	\begin{proof}
	Let $\tilde v,\tilde q\in \mathbb L_h(\Kt)$. 
Since functions in $\mathbb L_h(\Kt)$ vanish on
	$\partial\Kt_1\times \Kt_2$, integration by parts yields
	\begin{align*}
	\bigl((\widetilde A_K-\AKO)\tilde v,\tilde q\bigr)_{\Kt}
	&=
	\eps^2\bigl((\nu_{G_K}-I)\widetilde\nabla \tilde v,\widetilde\nabla \tilde q\bigr)_{\Kt}
	+\bigl((\rho_{G_K}-1)\tilde v,\tilde q\bigr)_{\Kt}
	\\
	&\quad
	-\eps^2 ( (\nu_{G_K}-I)\widetilde\nabla \tilde v\cdot \bn_{\Kt}, \tilde q)_{\Kt_1\times \partial\Kt_2}.
	\nonumber
	\end{align*}
	
	By expanding the matrix-vector product componentwise and using the H\"older inequality, we get
	\begin{align*}
	\bigl|\big((\nu_{G_K}-I)\widetilde\nabla \tilde v,\widetilde\nabla \tilde q\big)_{\Kt}\bigr|
	&\le
	\|\eta_{11}\|_{L^\infty(\Kt)}
	\|\partial_1\tilde v\|_{\Kt}
	\|\partial_1 \tilde q\|_{\Kt}
	+
	\|\eta_{12}\|_{L^\infty(\Kt)}
	\|\partial_2\tilde v\|_{\Kt}
	\|\partial_1 \tilde q\|_{\Kt}
	\\
	&\quad +
	\|\eta_{21}\|_{L^\infty(\Kt)}
	\|\partial_1\tilde v\|_{\Kt}
	\|\partial_2 \tilde q\|_{\Kt}
	+
	\|\eta_{22}\|_{L^\infty(\Kt)}
	\|\partial_2\tilde v\|_{\Kt}
	\|\partial_2 \tilde q\|_{\Kt}.
	\end{align*}
	
    Using the anisotropic inverse estimate \eqref{ainv} for both $\tilde v$ and $\tilde q$, we obtain
	\begin{align}
	\label{eq:interior-aniso}
	\eps^2\bigl|\big((\nu_{G_K}-I)\widetilde\nabla \tilde v,\widetilde\nabla q\big)_{\Kt}\bigr|
	&\lesssim
	\eps^2
	\Bigl(
	\|\eta_{11}\|_{L^\infty(\Kt)}
	+\frac{\hKot}{\hKtt}\|\eta_{12}\|_{L^\infty(\Kt)}\\
	&
	\quad +\frac{\hKot}{\hKtt}\|\eta_{21}\|_{L^\infty(\Kt)}
	+\left(\frac{\hKot}{\hKtt}\right)^2\|\eta_{22}\|_{L^\infty(\Kt)}
	\Bigr)
	\|\partial_1\tilde v\|_{\Kt}
	\|\partial_1 \tilde q\|_{\Kt}.
	\nonumber
	\end{align}
	
	The zeroth-order term is estimated 
	by
	\begin{equation}\label{eq:rho-aniso}
	\bigl|\big((\rho_{G_K}-1)\tilde v, \tilde q\big)_{\Kt}\bigr|
	\le
	\|\rho_{G_K}-1\|_{L^\infty(\Kt)}
	\|\tilde v\|_{\Kt}
	\|\tilde q\|_{\Kt}.
	\end{equation}
	
	It remains to estimate the boundary term. Since the outward normal on
	$\Kt_1\times \partial\Kt_2$ is
	$\bn_{\Kt}=\pm(0,1)^\top$, we have
	\begin{equation*}
	|(\nu_{G_K}-I)\widetilde\nabla \tilde v\cdot \bn_{\Kt}|
	=
	|\eta_{21} \partial_1 \tilde v+\eta_{22} \partial_2\tilde v|
	\qquad\text{on } \Kt_1 \times \partial\Kt_2.
	\end{equation*}
	By the anisotropic trace and inverse inequalities~\eqref{eq:anisotropictraceineq}--\eqref{ainv}, together with the Poincar\'e--Steklov inequality~\eqref{eq:SteklovPoincare}, we obtain
	\begin{align}
	\label{eq:boundary-aniso}
	\eps^2\bigl|\big( (\nu_{G_K}-I)\widetilde\nabla \tilde v\cdot \bn, \tilde q\big)_{\Kt_1\times \partial\Kt_2}\bigr|
	\lesssim
	\eps^2
	\Bigg(
	    \frac{\hKot}{\hKtt}\|\eta_{21}\|_{L^\infty(\Kt)}
	    +\left(\frac{\hKot}{\hKtt}\right)^2\|\eta_{22}\|_{L^\infty(\Kt)}
	\Bigg)
	\|\partial_1\tilde v\|_{\Kt}
	\|\partial_1 \tilde q\|_{\Kt}.
	\end{align}
	
	Combining \eqref{eq:interior-aniso}, \eqref{eq:rho-aniso}, and \eqref{eq:boundary-aniso},
	and using the definition~\eqref{eq:norm_Lh} of the $\IL_h(\Kt)$ norm, we conclude that there exists a constant $C_{\rm pert}>0$, independent of  $\hKot$, $\hKtt$, and $\eps$ such that
	\begin{equation}\label{eq:perturb-aniso-bound}
	\|(\widetilde A_K-\AKO)\tilde v\|_{\mathbb L_h(\Kt)'}
	\le
	C_{\rm pert}\,\theta_K\,
	\|\tilde v\|_{\mathbb L_h(\Kt)}
	\qquad\forall \tilde v\in \mathbb L_h(\Kt).
	\end{equation}
	
	By \cref{lem:localcoercivity}, 
	$ \|\AKO \tilde v\|_{\mathbb L_h(\Kt)'} \ge c_{A,0}\|\tilde v\|_{\mathbb L_h(\Kt)}$ for all $\tilde v\in\mathbb L_h(\Kt)$,
	which, combined with \eqref{eq:perturb-aniso-bound}, implies \eqref{eq:prototype} with $\gamma_K=C_{\rm pert}\theta_K c_{A,0}^{-1}$. 
	Therefore, by 
    the assumption~$\theta_K\leq c_{A,0} C_{\rm pert}^{-1}$, 
	\cref{lem:abstractneumann} applies and yields the result.
	\end{proof}

	\begin{remark}[When is $\theta_K$ small?]\label{rem:Theta-small}
	The quantity $\theta_K$ measures the deviation of the map $G_K$ from a rigid motion, 
	resolved
	 in the anisotropic coordinates of the pullback element $\Kt$.
	Recall that the element stretching is already encoded in the affine map $F_K$, while $G_K$ describes only the remaining geometric deformation.
	Therefore, the relevant regime for the analysis is that $G_K$ is a uniformly mild bending of the tensor-product element $\Kt$, as expressed in~\eqref{eq:GK}.
	
    In particular, if $G_K$ is a rigid motion, i.e.,\ $ G_K(\tilde \bx) = Q_K\tilde \bx+b$ with $Q_K\in\mathbb R^{2\times 2}$ an orthogonal matrix and $b_K\in\mathbb R^2$, then $ \nu_{G_K}=I$ and $\rho_{G_K}=1$.
	Hence $ \theta_K=0$, and the lemma applies without any smallness condition.
	
	Moreover, the anisotropic weighting in $\theta_K$ shows that the smallness condition becomes weaker on strongly stretched elements, whenever the deformation acts mainly through the long direction $\tilde x_2$.
	The first term in $\theta_K$ measures the deviation of the Jacobian determinant $\det(J_{G_K})$ from a constant, i.e.,\ the change in volume under the mapping $G_K$.
	As 
    $\GK$ is only meant to bend the element without stretching, this term is expected to be small.
	The remaining terms in $\theta_K$ measure how the Laplacian transforms under the mapping $G_K$. 
	As we have seen, the dominant term in the Laplacian is $\partial_1^2$, hence the smallness condition on $\eta_{11}$ is not weighted, whereas the other terms are weighted by powers of the anisotropic ratio $\frac{\hKot}{\hKtt}$.
	\end{remark}
	
	\begin{corollary}\label{cor:localcoercivity-physical}
	Assume that  $\theta_K$ defined in~\eqref{eq:thetaK-aniso} satisfies
	$ \theta_K<\frac{c_{A,0}}{C_{\rm pert}}$, 	
		with  $c_{A,0}$ as in Lemma~\ref{lem:localcoercivity} and $C_{\rm pert}>0$ as in~\eqref{eq:perturb-aniso-bound}.
	Then there exists a constant $c_A>0$, independent of $\hKot$, $\hKtt$, and~$\eps$, but depending on the uniform geometric constants in~\eqref{eq:GK} and on $p$, such that
	\begin{equation*}
	\|A_K v\|_{\LKh'} \ge c_A \|v\|_{\Vh} \qquad\forall v\in \LKh.
	\end{equation*}
	\end{corollary}
	\begin{proof}
	Let $v\in \LKh$ and set $\tilde v:=v\circ G_K\in \mathbb L_h(\Kt)$.
	Using \cref{lem:perturb-bent-aniso} together with \eqref{ainv} and the anisotropy assumption \eqref{eq:anisotropicity}, we have that 
	\begin{equation*}
	\|\widetilde A_K \tilde v\|_{\mathbb L_h(\Kt)'}^2
	\gtrsim
	\|\tilde v\|_{\Kt}^2+\eps^2\|\widetilde\nabla \tilde v\|_{\Kt}^2.
	\end{equation*}
	
	For any 
	edge $\tilde e\subset \Kt_1\times \partial \Kt_2$, the anisotropic trace inequality~\eqref{eq:anisotropictraceineq}
	and the Poincar\'e--Steklov inequality~\eqref{eq:SteklovPoincare} yield
	\begin{equation*}
	\|\tilde v\|_{\tilde e}^2
	\lesssim
	\hKtt^{-1}\|\tilde v\|_{\Kt}^2
	\lesssim
	\hKtt^{-1}\hKot^2\|\partial_1 \tilde v\|_{\Kt}^2
	\lesssim
	\hKtt \| \partial_1 \tilde v\|_{\Kt}^2,
	\end{equation*}
	since $\hKot\le \hKtt$.
	Combining the two previous estimates and using that  $\tilde v\in \mathbb L_h(\Kt)$ vanishes on
			$\partial \Kt_1\times \Kt_2$ gives
	\begin{equation}\label{eq:pullback-DG}
	\|\widetilde A_K \tilde v\|_{\mathbb L_h(\Kt)'}^2
	\gtrsim
	\|\tilde v\|_{\Kt}^2
	+\eps^2\|\widetilde\nabla \tilde v\|_{\Kt}^2
	+\sum_{\tilde e\subset\partial\Kt}
	\frac{\eps^2}{
	\hKtt}\|\tilde v\|_{\tilde e}^2.
	\end{equation}
	
	Finally, the uniform bounds~\eqref{eq:GK} imply the norm equivalences
	\begin{equation*}
	\|v\|_{K}^2 \lesssim \|\tilde v\|_{\Kt}^2,
	\qquad
	\|\nabla v\|_{K}^2 \lesssim \|\widetilde\nabla \tilde v\|_{\Kt}^2,
	\qquad
	\|v\|_{e}^2 \lesssim \|\tilde v\|_{\tilde e}^2,
	\end{equation*}
	for every edge $e\subset\partial K$ and its preimage $\tilde e=G_K^{-1}(e)$.
	Using that $\hKtt\ge h_\perp^{\tilde e,\Kt}$ on $e \subset \partial\Kt$ since $\hKot\leq \hKtt$, and recalling that
		$
		\|A_K v\|_{\LKh'}=\|\widetilde A_K \tilde v\|_{\mathbb L_h(\Kt)'}
		$ (since $\LKh$ is endowed with the pullback norm
		$
		\|v\|_{\LKh}=\|\tilde v\|_{\mathbb L_h(\Kt)}
		$), estimate \eqref{eq:pullback-DG} yields
	\begin{equation*}
	\|A_K v\|_{\LKh'}^2 \gtrsim
	\|v\|_{K}^2
	+\eps^2\|\nabla v\|_{K}^2
	+\sum_{e\subset\partial K}\frac{\eps^2}{h_\perp^{e,\Kt}}\|v\|_{e}^2,
	\end{equation*}
	which is the claim.
	\end{proof}
	
	\subsubsection{Local continuity} \label{sec:localcontinuity}
	\begin{theorem}[Continuity of $A_K$ on $\Vsh$] \label{thm:continuity-AK}
	There exists a constant $C_A>0$, depending only on the polynomial degree $p$ and on the geometric constants in~\eqref{eq:GK}, such that
	\begin{equation*}
	\sum_{K\in\calK} \|\AK u\|_{\LKh'}^2 \le C_A^2 \|u\|_{\Vsh}^2
	\qquad\forall u\in\Vsh.
	\end{equation*}
	\end{theorem}
	\begin{proof}
	Let $K\in\calK$ and let $v\in \LKh$. Denote by $\widetilde v:=v\circ G_K$ the pullback of $v$ to $\Kt$.
	Since $v$ vanishes on $\calE_2^K= G_K(\partial \Kt_1\times \Kt_2)$,
	Integration by parts on $K$ yields
	\begin{equation}\label{eq:cont-AK-ibp}
	\inner{A_Ku,v}_{K}
	=
	(\eps^2 \nabla u,\nabla v)_{K}
	+
	(u,v)_{K}
	-
	\sum_{e\in\calE^K_1}
	( \eps^2 \nabla u\cdot \bn_K,v)_{e}.
	\end{equation}
	We apply the Cauchy--Schwarz inequality to each term on the right-hand side and estimate the resulting 
	$v$-terms in the $\LKh$ norm.

	Using the properties~\eqref{eq:GK} of $G_K$, Lemma~\ref{lem:anisotropic-inverse}, and the anisotropy assumption~\eqref{eq:anisotropicity} ($\hKot\le \hKtt$), we obtain
	\begin{align*}
	\eps^2 \|\nabla v\|_{K}^2
	\lesssim
	\eps^2 \|\widetilde\nabla \widetilde v\|_{\Kt}^2
	\lesssim
	\eps^2 \|\partial_1 \widetilde v\|_{\Kt}^2
	\le
	\norm{v}_{\LKh}^2.
	\end{align*}
	Let $e\in\calE^K_1$ and $\widetilde e:=G_K^{-1}(e)\subset \Kt_1\times \partial \Kt_2$.
	Using again~\eqref{eq:GK} and~\eqref{eq:anisotropicity}, together with the anisotropic trace inequality~\eqref{eq:anisotropictraceineq}, and the Poincar\'e--Steklov  inequality~\eqref{eq:SteklovPoincare}, we get
	\begin{equation*}
	\norm{v}_{e}^2
	\lesssim
	\norm{\widetilde v}_{\widetilde e}^2
	\lesssim
	\hKtt^{-1}\norm{\widetilde v}_{\Kt}^2
	\lesssim
	\hKtt^{-1}\hKot^2 \norm{\partial_1 \widetilde v}_{\Kt}^2
	\lesssim
	\hKtt \norm{\partial_1 \widetilde v}_{\Kt}^2
	\lesssim
	\hKtt\eps^{-2}\norm{v}_{\LKh}^2.
	\end{equation*}
	Recalling that, for $e\in\calE^K_1$, the adjacent edge length is $ h_\perp^{e,\Kt}=\hKtt$ and combining  \eqref{eq:cont-AK-ibp} with the above bounds, we derive
	\begin{align*}
	\inner{A_Ku,v}_{K}
	&\lesssim
	\Big(
	\norm{\eps\nabla u}_{K}^2
	+
	\norm{u}_{K}^2
	+
	\sum_{e\in\calE^K_1}
	h_\perp^{e,\Kt}\norm{\eps\nabla u\cdot \bn_K}_{e}^2
	\Big)^{1/2}
	\norm{v}_{\LKh}.
	\end{align*}
	Taking the supremum over $v\in\LKh$ and summing over all $K\in\calK$, we conclude
	\begin{equation*}
	\sum_{K\in\calK}\norm{A_Ku}_{\LKh'}^2
	\lesssim
	\sum_{K\in\calK}
	\Big(
	\norm{\eps\nabla u}_{K}^2
	+
	\norm{u}_{K}^2
	+
	\sum_{e\subset \partial K}
	h_\perp^{e,\Kt}\norm{\eps\nabla u\cdot \bn_K}_{e}^2
	\Big)
	\lesssim
	 \|u\|_{\Vsh}^2,
	\end{equation*}
    as desired.
	\end{proof}
	
	\subsection{Global problem}\label{sec:global}
	We establish discrete coercivity and continuity of the bilinear form $a_h(\cdot,\cdot)$ in Theorem~\ref{thm:global}.
	We begin with bounds for the facet terms.
	\begin{lemma}[Facet estimates]\label{lem:facet-est}
	For all $u,v\in \Vsh$, it holds
	\begin{align}
	\Big|\sum_{e\in\calE^{\calI}\cup \calE^\calD}\int_e \mvl{\eps^2 \nabla v}_{\hperp^e}\cdot \jmp{
	u}\Big|
	&\lesssim
	\Bigg(
	\sum_{K\in\calK}\sum_{e\subset\partial K}
	h_\perp^{e,\Kt}
	\|\eps \nabla v\cdot \bn_e\|_{e}^2
	\Bigg)^{1/2} 	
	|
	u|_{\mathrm{J}}.
	\label{eq:facet-est-1}
	\end{align}
	Moreover, for all $u\in \Vsh$ and $v_h\in \Vh$,  it holds
	\begin{align}
	\Big|
	\sum_{e\in\calE^{\calI} \cup \calE^\calD}\int_e \mvl{\eps^2 \nabla v_h}_{\hperp^e}\cdot \jmp{u}\Big|
	&\le
	C_{\mathrm{facet}}
	\Bigg(
	\sum_{K\in\calK}\eps^2\|\nabla v_h\|_{K}^2
	\Bigg)^{1/2}\abs{u}_{\mathrm{J}},
	\label{eq:facet-est-2}
	\end{align}
	where $C_{\mathrm{facet}}$ depends only on $p$ and on the geometric constants in~\eqref{eq:GK}.
	\end{lemma}
	\begin{proof}
	Let $u,v\in \Vsh$.
	For all $e\in\calE^{\calI}$ with $e=\partial K^+\cap \partial K^-$, 
	the definition of the weighted average~\eqref{eq:averages-jumps} and two applications of the Cauchy--Schwarz inequality give
	\begin{align*}
	\begin{split}
			\int_e  &\mvl{\eps^2\nabla v}_{\hperp^e} \cdot \jmp{u}  
	         = \int_e \eps^2 \left(\frac{h^+_\perp}{h^+_\perp+h^-_\perp}\nabla v|_{K^+}\cdot \bn_e + \frac{h^-_\perp}{h^+_\perp+h^-_\perp}\nabla v|_{K^-}\cdot \bn_e\right)\cdot \jmp{u}\\
	        & \leq 
	        \left( h^+_\perp\eps^2\norm{\nabla v{|_{K^+}}\cdot \bn_e}_{e}^2 + h^-_\perp \eps^2 \norm{\nabla v{|_{K^-}}\cdot \bn_e}_{e}^2 \right)^{\frac12} \left(h^+_\perp+h^-_\perp\right)^{-\frac12}\norm{\eps\jmp{u}}_{e}.
	    \end{split}
	\end{align*}
	For all $e\in\calE^{\calD}$ with $e\subset\partial K$, we have $\int_e  \eps^2 (\nabla v\cdot \bn_\Omega) u \leq 
	h_\perp^\frac12 \eps\norm{\nabla v{|_{K}}\cdot \bn_e}_{e} \left(h_\perp\right)^{-\frac12}\norm{\eps\jmp{u}}_{e}$.
	Summing over all facets and collecting elementwise contributions yields~\eqref{eq:facet-est-1}.
	Estimate~\eqref{eq:facet-est-2} follows from~\eqref{eq:facet-est-1} by taking $v=v_h\in \Vh$ and applying the anisotropic inverse inequality~\eqref{eq:anisotropictraceineq} on the pullback element, together with assumption~\eqref{eq:GK} of $G_K$.
	\end{proof}
	\begin{theorem}[Coercivity and continuity of $a_h$]\label{thm:global}
    For all $\sigma \geq \sigma_0$ for a fixed $\sigma_0 > C_{\mathrm{facet}}^2/2$, with $C_{\mathrm{facet}}$ as in~\eqref{eq:facet-est-2}, it holds
	\begin{equation*}
	a_h(v_h,v_h)\ge c_a \|v_h\|_{\Vh}^2
	\qquad\forall v_h\in\Vh,
	\end{equation*}
	for some constant $c_a>0$ independent of $\eps$ and of $\hKot$ and $\hKtt$ for all $K\in\calK$.
	Moreover, there exists a constant $C_a>0$,  independent of $\eps$ and of $\hKot$  and $\hKtt$ for all $K\in\calK$, such that
		\begin{equation*}
	|a_h(u,v_h)|\le C_a \|u\|_{\Vsh}\|v_h\|_{\Vh}
		\qquad\forall (u,v_h)\in \Vsh\times \Vh.
		\end{equation*}
	\end{theorem}
	\begin{proof}
	Let $v_h\in\Vh$. Expanding $a_h(v_h,v_h)$, using~\eqref{eq:facet-est-2} and the Young inequality, we obtain
	\begin{align*}
	a_h(v_h,v_h)
	&=
	\sum_{K\in\calK}\Big(\eps^2\|\nabla v_h\|_{K}^2+\|v_h\|_{K}^2\Big)
	+\sigma \abs{v_h}_{\mathrm{J}}^2
	\\
	&\quad
	-2\sum_{e\in\calE^{\calI}}\int_e \mvl{\eps^2\nabla v_h}_{\hperp^e}\cdot \jmp{v_h}
	-2\sum_{e\in\calE^{\calD}}\int_e \eps^2(\nabla v_h\cdot \bn_\Omega)\,v_h
	\\
	&\ge
	\sum_{K\in\calK}\Big(\eps^2\|\nabla v_h\|_{K}^2+\|v_h\|_{K}^2\Big)
	-C_{\mathrm{facet}} \abs{v_h}_{\mathrm{J}}
	\Bigg(\sum_{K\in\calK}\eps^2\|\nabla v_h\|_{K}^2\Bigg)^{1/2}
	+\sigma \abs{v_h}_{\mathrm{J}}^2
	\\
	& \ge\frac{1}{2}\sum_{K\in\calK}\eps^2\|\nabla v_h\|_{K}^2
	+\sum_{K\in\calK}\|v_h\|_{K}^2
	+\left(\sigma - \frac{C_{\mathrm{facet}}^2}{2}\right)|v_h|_{\mathrm{J}}^2.
	\end{align*}
	The choice of the penalty parameter $\sigma > C_{\mathrm{facet}}^2/2$ concludes the proof of coercivity.

	Let $(u,v_h)\in \Vsh\times \Vh$. Using the Cauchy--Schwarz inequality, the definition~\eqref{eq:Vhnorm} of $\|\cdot\|_{\Vsh}$,
	and the bounds~\eqref{eq:facet-est-1}--\eqref{eq:facet-est-2}, we obtain
	\begin{align*}
	|a_h(u,v_h)|
	& \lesssim
	\sum_{K\in\calK}
	\Big(
	\eps^2\|\nabla u\|_{K}\|\nabla v_h\|_{K}
	+\|u\|_{K}\|v_h\|_{K}
	\Big)
	+\sigma |u|_{\mathrm{J}} |v_h|_{\mathrm{J}}
	\\
	&\quad
	+\Bigg(
	\sum_{K\in\calK}\sum_{e\subset\partial K}
	h_\perp^{e,\Kt}
	\|\eps \nabla u\cdot \bn_e\|_{e}^2
	\Bigg)^{1/2}
	|v_h|_{\mathrm{J}}
	+ |u|_{\mathrm{J}}
	\Bigg(\sum_{K\in\calK}\eps^2\|\nabla v_h\|_{K}^2
	\Bigg)^{1/2}
	\\
	&\lesssim \|u\|_{\Vsh}\|v_h\|_{\Vh},
	\end{align*}
    which completes the proof of continuity.
	\end{proof}
	
	\subsection{Well-posedness and quasi-optimality}\label{sec:wp-qo}
	In the following theorem, we establish well-posedness of the discrete problem~\eqref{eq:PDEh}, by applying \cite[Theorem 3.2]{LLSV_ARXIV_2024}, since the local and global properties
    \eqref{eq:A_coerc}--\eqref{eq:ah_cont} have been verified in the previous sections.
	 Quasi-optimality is derived from \cite[Corollary 3.3]{LLSV_ARXIV_2024}.
	\begin{theorem}[Well-posedness and quasi-optimality]
	\label{thm:wp-qo}
	Assume~\eqref{eq:anisotropicity}, the geometric smallness condition~$ \theta_K<\frac{c_{A,0}}{C_{\rm pert}}$ as in Lemma~\ref{lem:perturb-bent-aniso}, 
    and that the penalty parameter~$\sigma 
    > C_{\mathrm{facet}}^2/2$,
	with $C_{\mathrm{facet}}$ as in~\eqref{eq:facet-est-2}.
	Then the embedded Trefftz DG problem~\eqref{eq:PDEh} is well posed.
	If, in addition, the solution~$u$ to~\eqref{eq:PDE} belongs to~$H^2(\Omega)$,
    then the discrete 
	solution $u_h\in\Vh$ satisfies
	\begin{equation*}
	\norm{u-u_h}_{\Vh}
	\le 
	C_{\rm qo}
	\inf_{v_h\in\Vh}\|u-v_h\|_{\Vsh},
	\end{equation*}
	for some constant $C_{\rm qo}>0$ independent of $\eps$ and of 
	  $h_{K,1}$ and $h_{K,2}$ for all $K\in\calK$.
	\end{theorem}
	\begin{proof}
	The local properties~\eqref{eq:A_coerc} and~\eqref{eq:A_cont} follow from
			Corollary~\ref{cor:localcoercivity-physical} and 
			Theorem~\ref{thm:continuity-AK}, respectively.
	Theorem~\ref{thm:global} yields coercivity and continuity of $a_h(\cdot, \cdot)$ 
	on $\Vh$ and $\Vsh\times\Vh$, respectively. Since $\ITh\subset\Vh$, 
	this gives, in particular, the global properties~\eqref{eq:acoerc} and~\eqref{eq:ah_cont}.
	Hence the problem~\eqref{eq:PDEh} is well posed by \cite[Theorem 3.2]{LLSV_ARXIV_2024}.
	
	Let $u\in H^2(\Omega)$ be the solution to~\eqref{eq:PDE}, then for every $K\in\calK$, 
	$v_\IL\in\LKh$, and $v_\IT\in\ITh$, we have
	\begin{equation*}
		(A_Ku,v_\IL)_{K} = (f,v_\IL)_{K} = \ell_K(v_\IL),
		\qquad
		a_h(u,v_\IT) = \ell_h(v_\IT),
	\end{equation*}
	where the first identity uses $-\eps^2\Delta u+u=f$, and the last identity follows by
	elementwise integration by parts and  the boundary condition in~\eqref{eq:PDE}. The quasi-optimal 
	estimate then follows from \cite[Corollary 3.3]{LLSV_ARXIV_2024}.
	\end{proof}
	
	\subsection{A priori error estimates}\label{sec:apriori}
	In this section, we write $h_1:=\hKot$ and $h_2:=\hKtt$, for brevity.
  	Let $\widetilde \Pi_K^p$ denote the $L^2(\widetilde K)$-orthogonal projection onto $\IP^p_\otimes(\widetilde K)$.  
    We define the pullback projection $\Pi_K^p:L^2(K)\to \Vh(K)$ by
    \[
    \Pi_K^p v := \bigl(\widetilde\Pi_K^p(v\circ G_K)\bigr)\circ G_K^{-1}\qquad \forall v \in L^2(K).
    \]
	We recall the approximation properties from \cite[Lemmas 7.5, 7.7, and 7.9]{G06hp} in \cref{lem:proj_estimates}.
	The anisotropic bounds are stated in terms of norms of derivatives of $\widetilde v$ on $\Kt$ rather than norms of derivatives of $v$ on $K$, since this yields sharper bounds that capture the different behaviour of $v$ in the actual directions of anisotropy.
	
\newcommand{\jj}{\bar\imath}
	\begin{theorem}[Error estimate in the energy norm]
		\label{thm:best_approximation}
        Assume that the solution $u$ to~\eqref{eq:PDE} belongs to $H^{\ell+1}(\calK)\cap H^2(\Omega)$, for some $\ell\ge 1$, and that 
		$G_K$ is a $C^{\ell+1}$-diffeomorphism for every $K\in\calK$.
		Let $u_h\in\Vh$ be the embedded Trefftz DG solution to~\eqref{eq:PDEh}.
        Set $\jj=3-i$, for $i=1,2$.
		Then, for $0\leq s\leq \min\{p,\ell\}$, the following bound holds
\begin{equation*}
\|u-u_h\|_{\Vh}^2
\lesssim
\sum_{K\in\calK}\sum_{i=1}^2 h_i^{2s} \left[ h_i^2 + \eps^2 \left( 1+\left(\frac{h_i}{h_{\jj}}\right)^2 \right) \right] |\widetilde u|_{H^{s+1}(\Kt),i}^2 .
\end{equation*}
		 where $0\leq s\leq \min\{p,\ell\}$ and 
\begin{equation*}
|\widetilde v|_{H^{s+1}(\Kt),i}^2 := \|\partial_i^{s+1}\widetilde v\|_{\Kt}^2 + \left(\frac{h_{\jj}}{h_i}\right)^2 \|\partial_i^s\partial_{\jj}\widetilde v\|_{\Kt}^2 .
\end{equation*}
	\end{theorem}
	
	\begin{proof}
		The bound in Theorem~\ref{thm:wp-qo}, gives
		\begin{align*}
			\|u-u_h\|_{\Vh}^2&	\lesssim\sum_{K\in\calK}\eps^2\|\nabla (u-v_h) \|_{K}^2+\sum_{K\in\calK}\|u-v_h\|_{K}^2
			+\sum_{K\in\calK}\sum_{e\subset\partial K}
			\frac{\eps^2}{h_\perp^{e,\Kt}}\|u-v_h\|_{e}^2\\&
		   +\sum_{K\in\calK}\sum_{e\subset\partial K}
			h_\perp^{e,\Kt}\|\eps\nabla (u-v_h) \cdot\bn\|_{e}^2
            =:  I_1+I_2+I_3+I_4 \qquad \forall v_h\in \Vh.
		\end{align*}	
		Let $\vh\in \Vh$ be defined as $v_h|_K  = \Pi^p_K (u|_K)$ for all $K\in \calK$.
		By \eqref{eq:proj_est_H1} and \eqref{eq:proj_est_L2}, the first two terms satisfy
		\begin{align*}
			I_1
			&\lesssim
			\eps^2\sum_{K\in\calK}\sum_{i=1}^2
			\Big(
			h_i^{2s}\|\partial_i^{s+1}\widetilde u\|_{\Kt}^2
			+h_{\jj}^{2s}\|\partial_{\jj}^s\partial_i\widetilde u\|_{\Kt}^2
			\Big),
			\quad
			I_2
			\lesssim
			\sum_{K\in\calK}\sum_{i=1}^2
			h_i^{2s+2}\|\partial_i^{s+1}\widetilde u\|_{\Kt}^2.
		\end{align*}
	For $e\in\calE^i_K$, by definition
	$h_\perp^{e,\Kt}=h_{\jj}$.
	Using the trace estimate 
    in~\eqref{eq:proj_est_trace}, we obtain
    \begin{align*}
    I_3 &\lesssim \eps^2 \sum_{K\in\calK}\sum_{i=1}^2 \Bigg[ h_{\jj}^{2s} \|\partial_{\jj}^{s+1}\widetilde u\|_{\Kt}^2 + h_i^{2s} \left(\frac{h_i}{h_{\jj}}\right)^2 \|\partial_i^{s+1}\widetilde u\|_{\Kt}^2 + h_i^{2s} \|\partial_i^s\partial_{\jj}\widetilde u\|_{\Kt}^2 \Bigg]
    \\
    &\lesssim \eps^2 \sum_{K\in\calK}\sum_{i=1}^2 h_i^{2s} \Bigg[ \left( 1+\left(\frac{h_i}{h_{\jj}}\right)^2 \right) \|\partial_i^{s+1}\widetilde u\|_{\Kt}^2 + \|\partial_i^s\partial_{\jj}\widetilde u\|_{\Kt}^2
    \Bigg].
    \end{align*}
    The normal-derivative trace estimate \eqref{eq:proj_est_dtrace}, together with
    $|\nabla w\cdot\bn|\le |\partial_1 w|+|\partial_2 w|$, yield
    \begin{align*}
    I_4 &\lesssim \eps^2 \sum_{K\in\calK}\sum_{i=1}^2 h_{\jj} \Bigg[ h_i^{2s-1} \left( \frac{h_i}{h_{\jj}} \|\partial_i^{s+1}\widetilde u\|_{\Kt}^2 + \frac{h_{\jj}}{h_i} \|\partial_i^s\partial_{\jj}\widetilde u\|_{\Kt}^2 \right)
    \\
    &\qquad + h_{\jj}^{2s-1} \|\partial_{\jj}^{s}\partial_i\widetilde u\|_{\Kt}^2 + h_{\jj}^{2s-1} \|\partial_{\jj}^{s+1}\widetilde u\|_{\Kt}^2 + \frac{h_i^{2s}}{h_{\jj}} \|\partial_i^s\partial_{\jj}\widetilde u\|_{\Kt}^2 \Bigg]
    \\
    &\lesssim \eps^2 \sum_{K\in\calK}\sum_{i=1}^2 h_i^{2s} \Bigg[ \|\partial_i^{s+1}\widetilde u\|_{\Kt}^2 + \left( 1+\left(\frac{h_{\jj}}{h_i}\right)^2 \right) \|\partial_i^s\partial_{\jj}\widetilde u\|_{\Kt}^2 \Bigg].
    \end{align*}
Collecting all contributions and using the definition of the anisotropic seminorm gives the result.
\end{proof}

\section{Numerical experiments}\label{sect::numericalexperiments}
We present numerical experiments to validate the theoretical results and illustrate additional features and limitations of the proposed method.
All computations are performed using \texttt{NGSolve} \cite{ngsolve} and \texttt{NGSTrefftz} \cite{ngstrefftz}. 
Replication data are available in \cite{gomez_macias_2026_20322919}.
Implementational details of the embedded Trefftz DG method are given in~\cite{LS_IJMNE_2023}.
We briefly describe the construction of a basis for the special test space $\LKh$ defined in~\eqref{def:Lspace}, which can be easily implemented exploiting the tensor-product structure of the pullback elements. 
We define
\begin{equation*}
\psi_j(y):=y(1-y)\,L_j(y),\qquad j=0,\dots,p-2,
\end{equation*}
where $L_j$ denotes the standard Legendre polynomials shifted and scaled to  $[0,1]$.
Then $\{\psi_j\}_{j=0}^{p-2}$ spans the subspace of $\mathbb{P}^p([0,1])$ with homogeneous endpoint conditions. 
A basis for $\LKh$ is obtained by tensorization with a standard polynomial 
basis in the remaining variable.

We have also considered alternative choices for the test space $\LKh$, such as the image of the underlying space under the operator $A_K$, i.e., $\LKh = A_K(\Vh(K))$, and also $\LKh = \Delta(\Vh(K))$.
These alternative choices are possible, however, they exhibit numerical instabilities on fine meshes or high polynomial degrees.
Moreover, on anisotropic elements, redefining $\LKh$ by swapping the roles of $x$ and $y$, i.e.,\ vanishing on the shorter edges, leads to suboptimal results, confirming the importance of the specific choice we make for
the test space.
In all the following experiments, we compare the standard DG method using the space $\Vh$ to the embedded Trefftz DG method using $\ITh$.
The stabilization parameter $\sigma$ is set to $\sigma=10$ for all 
facets.
The numerical results are compared in terms of $h$-convergence on a square geometry 
in \cref{sec:RectangularElements}, and on curvilinear meshes 
for a circular domain in \cref{sec:CurvilinearElements}.  
\cref{sec:nineelements} investigates $hp$-convergence on a layer-adapted nine-element mesh, while \cref{sec:anisotropicdiffusion} tests the method on an anisotropic diffusion problem, which lies outside the theoretical framework.

\subsection{\texorpdfstring{$h$}{h}-convergence: square domain}\label{sec:RectangularElements}
We consider problem~\eqref{eq:PDE} on the square domain $\Omega=(-1,1)^2$, with diffusion parameter $\eps^2=10^{-5}$, 
homogeneous Dirichlet boundary conditions ($g= 0$), and source term $f$ such that the exact solution is
\begin{equation}\label{eq:solnineelemmesh}
	u(x,y)=\left(1-\frac{\mathrm{cosh}(x/\eps)}{\mathrm{cosh}(1/\eps)}\right)\left(1-\frac{\mathrm{cosh}(y/\eps)}{\mathrm{cosh}(1/\eps)}\right).
\end{equation}
A similar example was investigated in~\cite[\S 9.1]{G06hp} and, in one dimension, in~\cite[\S 6.1]{SS96bh}.
The solution exhibits boundary layers of thickness $\mathcal{O}(\eps)$ near $\partial \Omega$.
To resolve these layers, we employ anisotropic tensor-product meshes refined towards the boundary, obtained as the Cartesian product of two graded one-dimensional meshes. Following~\cite[eq.~(6.5)]{SS96bh}, the one-dimensional mesh is defined by $\{-1,x_1,\ldots,x_{2n},1\}$, where
\begin{equation*}
    x_i=-1-\delta
    \ln\Big(1-c\frac{i}{n+1}\Big), \quad 
    x_{n+i}=1+
    \delta
    \ln\Big(1-c\frac{n-i}{n+1}\Big), \quad i=1,\ldots,n,
\end{equation*}
with $c = 1-e^{-1/ \delta}$ and $\delta=\eps(p+0.5)3/\exp(1)$.

In \cref{fig:hconv}, we report results for increasing values of $n=1,2,\dots$ and polynomial degrees~$p=3,4,5$.
We plot the error, measured in the $\|\cdot\|_{\Vh}$-norm, against the square root of the total number of degrees of freedom. 
Both methods exhibit optimal convergence rates in terms of the  number of degrees of freedom.
Nonetheless, the embedded Trefftz method requires significantly fewer degrees of freedom to reach the same accuracy as the standard DG method.

 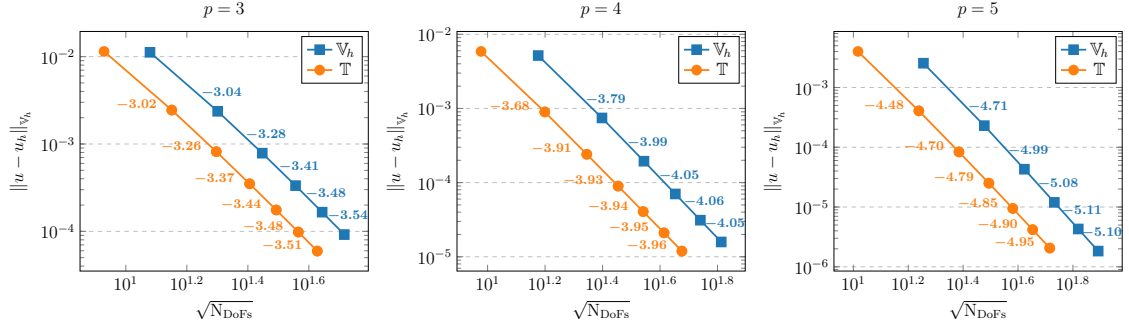
\begin{figure}[ht!]
		\centering
	    \resizebox{\textwidth}{!}{
	        \begin{tikzpicture}
					\begin{groupplot}[%
						group style={%
							group name={my plots},
							group size=3 by 1,
							horizontal sep=6em,
							vertical sep=5em,
						},
						legend style={
							legend columns=1,
							at={(0.98,0.98)},
						},
						ymajorgrids=true,
						grid style=dashed,
						cycle list name=colorsh,
	                    ymode=log, xmode=log
						]
	    	\CycleNextGruoupPloth{3}{dgeoc}
	    	\CycleNextGruoupPloth{4}{dgeoc}
	  	 	\CycleNextGruoupPloth{5}{dgeoc}
					\end{groupplot}   
			\end{tikzpicture}
	    }
	    \caption{$h$-convergence for the problem with exact solution $u$ in~\eqref{eq:solnineelemmesh}.
	       Comparison between the standard DG method ($\Vh$) and the embedded Trefftz DG method ($\ITh$) for polynomial degrees $p=3,4,5$.
	    }
	    \label{fig:hconv}
    \vspace{-.5em}
	\end{figure}
	
\subsection{\texorpdfstring{$h$}{h}-convergence: circular domain}\label{sec:CurvilinearElements}
We consider problem~\eqref{eq:PDE} on the unit disk with $\eps^2=10^{-5}$, homogeneous Dirichlet boundary conditions ($g=0$), and choose the source term $f$ such that the exact solution is
\begin{equation}\label{eq:circlesol}
	u(x,y)
	=
	1-\frac{\cosh\!\left(\sqrt{x^2+y^2}/\eps\right)}
	        {\cosh(1/\eps)} .
\end{equation}
This radially symmetric solution satisfies $u=0$ on $\partial\Omega$ and exhibits a boundary layer near~$\partial \Omega$. 

To resolve this layer, we use curved anisotropic meshes refined towards $\partial\Omega$.
The mesh consists of a fixed triangular mesh of the inner disk $B_{1/2}(0)$ and a fixed number of curved quadrilateral layers in the outer part $B_1(0)\setminus B_{1/2}(0)$.
We then refine by adding
layers of curved quadrilaterals towards the boundary, where the radial nodes added are given by the graded one-dimensional mesh used in the previous example $\{0,x_{n+1},\dots,x_{2n}\}$, with $n$ the number of refinement steps.

\begin{figure}[ht!]
\begin{minipage}{0.25\textwidth}
    \vspace{-1em}
  \includegraphics[width=\textwidth, trim={0cm 0cm 0cm 0cm}, clip]{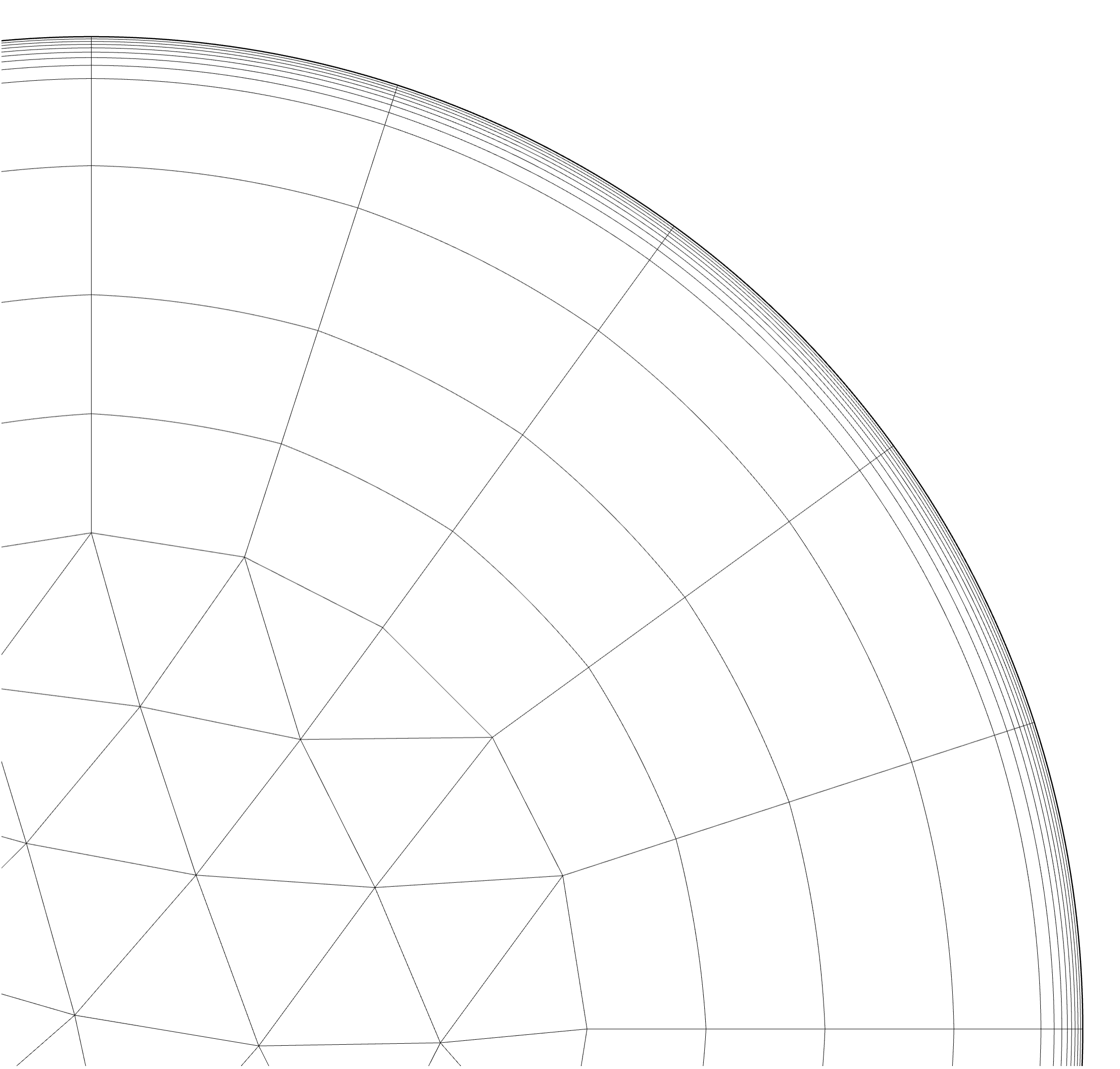}
    \\\vspace{1em}
  \includegraphics[width=\textwidth, trim={0cm 0cm 0cm 0cm}, clip]{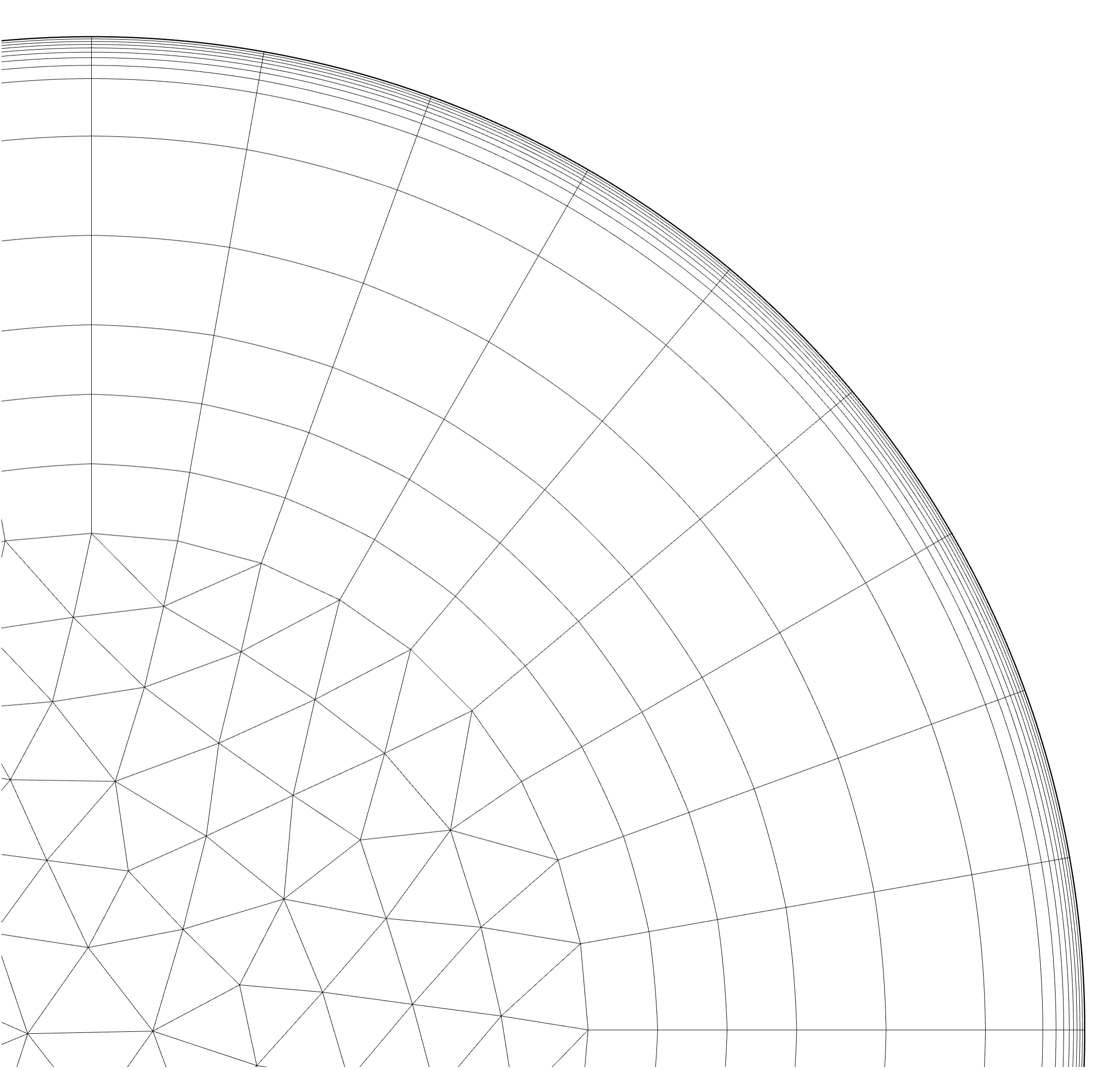}
\end{minipage}
\begin{minipage}{0.7\textwidth}
	\centering
    \resizebox{\textwidth}{!}{
        \begin{tikzpicture}
				\begin{groupplot}[%
					group style={%
						group name={my plots},
						group size=2 by 2,
						horizontal sep=6em,
						vertical sep=6em,
					},
					legend style={
						legend columns=1,
						at={(0.98,0.98)},
					},
					ymajorgrids=true,
					grid style=dashed,
					cycle list name=colorshp,
                    ymode=log, xmode=log
					]
    \nextgroupplot[ylabel={$\norm{u-u_h}_{\Vh}$},xlabel={$\mathrm{N}_{\mathrm{DoFs}}$},title={$p=4$}, title style={font=\large},
    label style={font=\large},
    tick label style={font=\large},
    legend style={font=\large}]
    \addplot+[discard if not={maxh}{0},discard if not={p}{4},discard if not={method}{dg}] table [x=ndof, y=dgerror, col sep=comma] {results/example2.csv};
    \addplot+[discard if not={maxh}{0},discard if not={p}{4},discard if not={method}{embt}] table [x=ndof, y=dgerror, col sep=comma] {results/example2.csv};
    \legend{$\Vh$,$\ITh$}
    \nextgroupplot[ylabel={$\norm{u-u_h}_{\Vh}$},xlabel={$\mathrm{N}_{\mathrm{DoFs}}$},title={$p=5$},title style={font=\large},
    label style={font=\large},
    tick label style={font=\large},
    legend style={font=\large}]
    \addplot+[discard if not={maxh}{0},discard if not={p}{5},discard if not={method}{dg}] table [x=ndof, y=dgerror, col sep=comma] {results/example2.csv};
    \addplot+[discard if not={maxh}{0},discard if not={p}{5},discard if not={method}{embt}] table [x=ndof, y=dgerror, col sep=comma] {results/example2.csv};
    \legend{$\Vh$, $\ITh$}
    \nextgroupplot[ylabel={$\norm{u-u_h}_{\Vh}$},xlabel={$\mathrm{N}_{\mathrm{DoFs}}$},title={$p=4$},title style={font=\large},
    label style={font=\large},
    tick label style={font=\large},
    legend style={font=\large}]
    \addplot+[discard if not={maxh}{1},discard if not={p}{4},discard if not={method}{dg}] table [x=ndof, y=dgerror, col sep=comma] {results/example2.csv};
    \addplot+[discard if not={maxh}{1},discard if not={p}{4},discard if not={method}{embt}] table [x=ndof, y=dgerror, col sep=comma] {results/example2.csv};
    \legend{$\Vh$,$\ITh$}
    \nextgroupplot[ylabel={$\norm{u-u_h}_{\Vh}$},xlabel={$\mathrm{N}_{\mathrm{DoFs}}$},title={$p=5$},   title style={font=\large},
    label style={font=\large},
    tick label style={font=\large},
    legend style={font=\large}]
    \addplot+[discard if not={maxh}{1},discard if not={p}{5},discard if not={method}{dg}] table [x=ndof, y=dgerror, col sep=comma] {results/example2.csv};
    \addplot+[discard if not={maxh}{1},discard if not={p}{5},discard if not={method}{embt}] table [x=ndof, y=dgerror, col sep=comma] {results/example2.csv};
    \legend{$\Vh$,$\ITh$}
				\end{groupplot}   
		\end{tikzpicture}
    }
\end{minipage}
    \caption{$h$-convergence for the problem with exact solution $u$ in~\eqref{eq:circlesol}.
        Comparison between the standard DG method ($\Vh$) and the embedded Trefftz method ($\ITh$) for polynomial degrees $p=4,5$ on two initial meshes with maximum element size $h\approx 0.15$ (top row) and $h\approx 0.09$ (bottom row).
		Refinement is performed anisotropically towards the boundary to resolve the boundary layers. The final mesh after eight refinement steps is shown on the left.
    }
    \label{fig:curved}
\vspace{-.5em}
\end{figure}
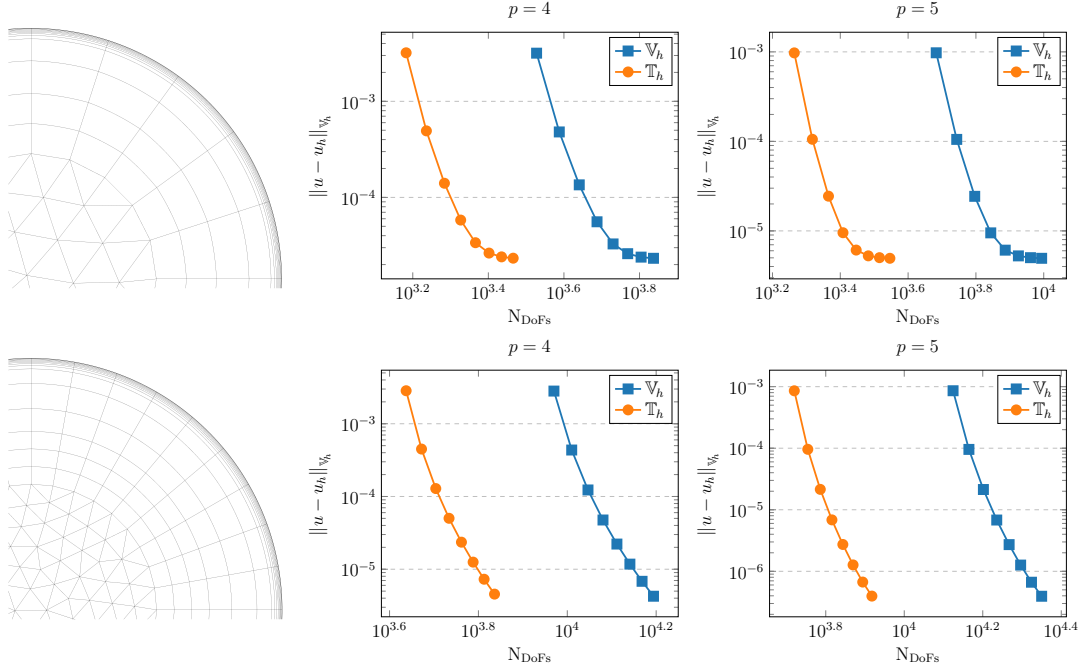

In the central region of the circle, we employ triangular elements solely to keep the mesh construction as simple as possible.
Although the theoretical analysis developed in this work does not concern itself with triangular elements, these elements are non-anisotropic and therefore fall within the framework of~\cite{LLSV_ARXIV_2024}, and are therefore easily covered by the same theoretical results.
On triangular elements, we use the standard Dubiner basis of total degree $p$, and the embedding is chosen as in~\cite{LLSV_ARXIV_2024}, i.e., we keep the local operator $A_K$ as is, but use the 
space of polynomials of degree $p-2$ as the test space.

\cref{fig:curved} shows two representative meshes obtained after eight refinement steps ($n=8$). On the right panels of the same figure, we report numerical results for polynomial degrees $p=4,5$ on successively refined meshes with $n=1,\dots,8.$
We plot the error in the energy norm against the total number of degrees of freedom.
Both methods exhibit similar behaviour: starting from a coarse initial triangular mesh, both show error stagnation due to insufficient radial resolution of the boundary layer, an effect not observed for sufficiently fine initial triangulations.
In all cases, the embedded Trefftz DG method outperforms the standard DG method,
achieving comparable accuracy with fewer
of degrees of freedom.

\subsection{\texorpdfstring{$hp$}{hp}-convergence}\label{sec:nineelements}
We revisit the problem on the square domain from \cref{sec:RectangularElements}, but now we investigate $hp$-convergence on a mesh with only nine elements.
We study the validity of the method for increasing $p$, despite the fact that the dependence on $p$ is not treated by our analysis explicitly.
Following \cite[Example 9.1]{G06hp}, we generate the mesh by partitioning the domain $\Omega=(-1,1)^2$ into a $3\times 3$ grid, see \cref{fig:9elem}.
To resolve the boundary layers, we perform $p$-refinements while simultaneously adapting the size of the small elements adjacent to the boundary.
The length of the small element edges is taken as $\ell = \lambda p \eps$, i.e., 
the increment of the polynomial degree $p$ is accompanied by a proportional increase in the size of the small elements.
We consider $\lambda=0.9$, which yields good convergence properties as reported in 
\cite{G06hp}, and $\lambda=0.54$, as suggested in
\cite{SS96bh}.

\begin{figure}[ht!]
\begin{center}    
\begin{minipage}{0.32\textwidth}
\vspace{-1em}
    \resizebox{\textwidth}{!}{
    \begin{tikzpicture}[scale=2.5, thick]
      \def\lam{0.2}
      \def\a{-1}
      \def\b{1}
      \draw[black] (\a,\a) rectangle (\b,\b);
      \foreach \x in {\a+\lam, \b-\lam} {
        \draw[black] (\x,\a) -- (\x,\b);
      }
      \foreach \y in {\a+\lam, \b-\lam} {
        \draw[black] (\a,\y) -- (\b,\y);
      }
      \node[below left] at (\a,\a) {$(-1,-1)$};
      \node[below right] at (\b,\a) {$(1,-1)$};
      \node[above left] at (\a,\b) {$(-1,1)$};
      \node[above right] at (\b,\b) {$(1,1)$};
      \draw[decorate,decoration={brace,mirror,raise=2pt}] 
        (\a,\a-0.05) -- (\a+\lam,\a-0.05)
        node[midway,below=4pt] {$\ell$};
      \draw[decorate,decoration={brace,raise=2pt}] 
        (-1-0.05,\a) -- (-1-0.05,\a+\lam)
        node[midway,left=4pt] {$\ell$};
    \end{tikzpicture}
}
\vspace{1em}
\resizebox{\textwidth}{!}{
    \begin{tikzpicture}[scale=2.5, thick]
      \def\lam{0.2}
      \def\a{-1}
      \def\b{1}
      \draw[black] (\a,\a) rectangle (\b,\b);
      \foreach \x in {\a+\lam, \b-\lam} {
        \draw[black] (\x,\a) -- (\x,\b);
      }
      \foreach \y in {\a+\lam, \b-\lam} {
        \draw[black] (\a,\y) -- (\b,\y);
      }
      \node[below left] at (\a,\a) {$(-1,-1)$};
      \node[below right] at (\b,\a) {$(1,-1)$};
      \node[above left] at (\a,\b) {$(-1,1)$};
      \node[above right] at (\b,\b) {$(1,1)$};
    \node at ({\a+0.5*\lam},{\a+0.5*\lam}) {$\IT_0$};
    \node at ({0},{\a+0.5*\lam}) {$\ITh$};
    \node at ({\b-0.5*\lam},{\a+0.5*\lam}) {$\IT_0$};
    \node at ({\a+0.5*\lam},0) {$\ITh$};
    \node at (0,0) {$\ITh$};
    \node at ({\b-0.5*\lam},0) {$\ITh$};
    \node at ({\a+0.5*\lam},{\b-0.5*\lam}) {$\IT_0$};
    \node at (0,{\b-0.5*\lam}) {$\ITh$};
    \node at ({\b-0.5*\lam},{\b-0.5*\lam}) {$\IT_0$};
    \end{tikzpicture}
}
\end{minipage}
\begin{minipage}{0.66\textwidth}
    \resizebox{\textwidth}{!}{
			\begin{tikzpicture}
				\begin{groupplot}[%
					group style={%
						group name={my plots},
						group size=2 by 2,
						horizontal sep=6em,
						vertical sep=5em,
					},
					legend style={
						legend columns=1,
                        at={(0.98,0.98)},
					},
					ymajorgrids=true,
					grid style=dashed,
					cycle list name=colorshp,
					]
    \nextgroupplot[ymode=log, ylabel={$\norm{u-u_h}_{\Vh}$},xlabel={$p$},legend pos=south west, title style={font=\large},
    label style={font=\large},
    tick label style={font=\large},
    legend style={font=\large}]
    \addplot+[discard if not={hnr}{1},discard if not={sigma}{0.54},discard if not={method}{dg}] table [x=p, y=dgerror, col sep=comma] {results/example3.csv};
    \addplot+[discard if not={hnr}{1},discard if not={sigma}{0.54},discard if not={method}{embt}] table [x=p, y=dgerror, col sep=comma] {results/example3.csv};
    \addplot+[dashed,discard if not={hnr}{1},discard if not={sigma}{0.9},discard if not={method}{dg}] table [x=p, y=dgerror, col sep=comma] {results/example3.csv};
    \addplot+[dashed,discard if not={hnr}{1},discard if not={sigma}{0.9},discard if not={method}{embt}] table [x=p, y=dgerror, col sep=comma] {results/example3.csv};          \legend{{$\Vh, \lambda=0.54$}, {$\ITh, \lambda=0.54$}, {$\Vh, \lambda=0.9$}, {$\ITh, \lambda=0.9$}}
    \nextgroupplot[ymode=log, ylabel={$\norm{u-u_h}_{\Vh}$},xlabel={$\sqrt{\mathrm{N}_{\mathrm{DoFs}}}$},xmin=4,xmax=69,  title style={font=\large},
    label style={font=\large},
    tick label style={font=\large},
    legend style={font=\large}]
    \addplot+[discard if not={hnr}{1},discard if not={sigma}{0.54},discard if not={method}{dg}] table [x=ndof12, y=dgerror, col sep=comma] {results/example3.csv};
    \addplot+[discard if not={hnr}{1},discard if not={sigma}{0.54},discard if not={method}{embt}] table [x=ndof12, y=dgerror, col sep=comma] {results/example3.csv};
    \addplot+[dashed,discard if not={hnr}{1},discard if not={sigma}{0.9},discard if not={method}{dg}] table [x=ndof12, y=dgerror, col sep=comma] {results/example3.csv};
    \addplot+[dashed,discard if not={hnr}{1},discard if not={sigma}{0.9},discard if not={method}{embt}] table [x=ndof12, y=dgerror, col sep=comma] {results/example3.csv};            \legend{{$\Vh, \lambda=0.54$}, {$\ITh, \lambda=0.54$}, {$\Vh, \lambda=0.9$}, {$\ITh, \lambda=0.9$}}
    \nextgroupplot[ymode=log, ylabel={$\norm{u-u_h}_{\Vh}$},xlabel={$p$},legend pos=south west,  title style={font=\large},
    label style={font=\large},
    tick label style={font=\large},
    legend style={font=\large}]
    \addplot+[discard if not={hnr}{2},discard if not={sigma}{0.54},discard if not={method}{dg}] table [x=p, y=dgerror, col sep=comma] {results/example3.csv};
    \addplot+[discard if not={hnr}{2},discard if not={sigma}{0.54},discard if not={method}{embt}] table [x=p, y=dgerror, col sep=comma] {results/example3.csv};
    \addplot+[dashed,discard if not={hnr}{2},discard if not={sigma}{0.9},discard if not={method}{dg}] table [x=p, y=dgerror, col sep=comma] {results/example3.csv};
    \addplot+[dashed,discard if not={hnr}{2},discard if not={sigma}{0.9},discard if not={method}{embt}] table [x=p, y=dgerror, col sep=comma] {results/example3.csv};          \legend{{$\Vh, \lambda=0.54$}, {$\ITh, \lambda=0.54$}, {$\Vh, \lambda=0.9$}, {$\ITh, \lambda=0.9$}}
    \nextgroupplot[ymode=log, ylabel={$\norm{u-u_h}_{\Vh}$},xlabel={$\sqrt{\mathrm{N}_{\mathrm{DoFs}}}$},xmin=4,xmax=69,  title style={font=\large},
    label style={font=\large},
    tick label style={font=\large},
    legend style={font=\large}]
    \addplot+[discard if not={hnr}{2},discard if not={sigma}{0.54},discard if not={method}{dg}] table [x=ndof12, y=dgerror, col sep=comma] {results/example3.csv};
    \addplot+[discard if not={hnr}{2},discard if not={sigma}{0.54},discard if not={method}{embt}] table [x=ndof12, y=dgerror, col sep=comma] {results/example3.csv};
    \addplot+[dashed,discard if not={hnr}{2},discard if not={sigma}{0.9},discard if not={method}{dg}] table [x=ndof12, y=dgerror, col sep=comma] {results/example3.csv};
    \addplot+[dashed,discard if not={hnr}{2},discard if not={sigma}{0.9},discard if not={method}{embt}] table [x=ndof12, y=dgerror, col sep=comma] {results/example3.csv};
    \legend{{$\Vh, \lambda=0.54$}, {$\ITh, \lambda=0.54$}, {$\Vh, \lambda=0.9$}, {$\ITh, \lambda=0.9$}}
				\end{groupplot}   
		\end{tikzpicture}
    }
    \end{minipage}
	\end{center}   
\vspace{-1.5em}
    \caption{$hp$-convergence for the problem with exact solution $u$ in~\eqref{eq:solnineelemmesh} on a nine-element mesh with adapting boundary layer size.
       Comparison between the standard DG method ($\Vh$) and the embedded Trefftz method ($\ITh$) for two different values of the small element size parameter $\lambda=0.54$ (solid lines) and $\lambda=0.9$ (dashed lines).
        On the left, we show the mesh used. In the top row, the Trefftz space used in the experiments is the one 
        described in the main text. 
        On the bottom row, the mixed choice of embedded Trefftz spaces is indicated in the mesh.
    }
    \label{fig:9elem}
\vspace{-.5em}
\end{figure}
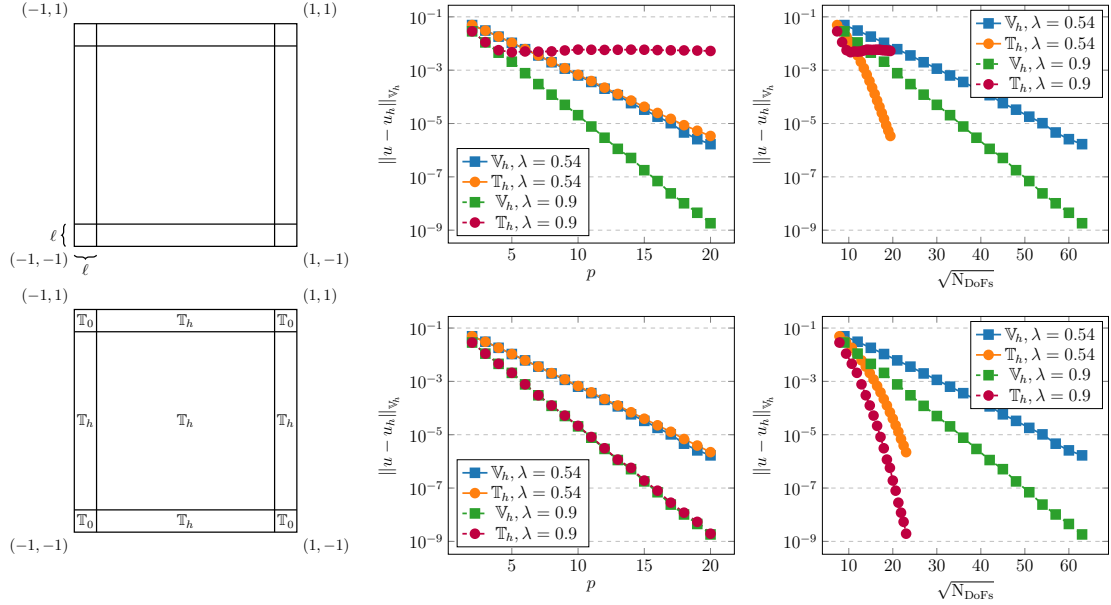

In \cref{fig:9elem} on the top row, we report the convergence results for both the methods, plotting the error in the $\|\cdot\|_{\Vh}$-norm against both, the polynomial degree $p$, and the square root of the number of degrees of freedom.
The standard DG method exhibits the expected exponential convergence in $p$ for both values of $\lambda$, with the case $\lambda=0.9$ showing a faster convergence rate due to the larger size of the small elements, which better capture the boundary layers.
In contrast, the embedded Trefftz method shows the same optimal rate for $\lambda=0.54$, but becomes unstable for $p \gtrsim 5$ for $\lambda=0.9$.
This is consistent with the fact that the present theoretical analysis is not $p$ robust, and does not guarantee stability for high polynomial degrees.

We have tested, and also
observed in the other experiments, that a smaller boundary layer or additional refinement 
regain stability, as
occurs for $\lambda=0.54$
in the example 
of \cref{sec:RectangularElements}.
The larger boundary layer seems to lead to an overfitting effect, as illustrated in \cref{fig:9elemsol} for $p=10$, where oscillations appear in the small elements located at the domain corners.
Motivated by this observation, we present another way to address this limitation, by modifying the discrete space in the corner elements by relaxing the Trefftz condition, allowing for a larger number of basis functions.
More precisely, in the corner elements, instead of testing with the tensor-product polynomials that vanish on two parallel edges of the element, we now test with tensor-product polynomials that vanish on all edges.
The new space is given by
\begin{equation*}
    \IT_0(K) := \{ v_h\in \Vh(K) \mid \inner{A_K v_h,q_h}_{K} = 0 \quad \forall q_h\in \Vh(K) \text{ with } q_h|_{\partial K}=0 \}.
\end{equation*}
This leads to a larger Trefftz space; 
since the constraint is not enforced on the edges, one has $\dim(\IT_0(K)) = 4p$.

The corresponding results are reported in the bottom row of \cref{fig:9elem}, using the combination of Trefftz spaces shown on the lower left panel of \cref{fig:9elemsol}.
With this modification, the embedded Trefftz method becomes stable and matches the exponential convergence in $p$ of the DG method.
Finally, we observe that this version of the embedded method yields comparable accuracy with a reduced number of degrees of freedom with respect to the standard DG method. 
We stress  that the additional cost introduced by the modified corner spaces is mild, as the overall number of DoFs still behaves linearly in $p$.
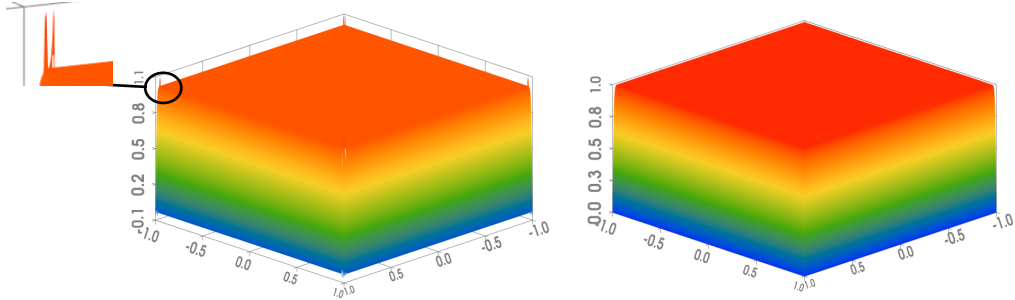
\begin{figure}[ht!]
  \centering
    \vspace{-.5em}
    \begin{tikzpicture}
        \node[anchor=south west, inner sep=0] (main) 
          {\includegraphics[width=0.4\textwidth, trim={3cm 0cm 3cm 2cm}, clip]{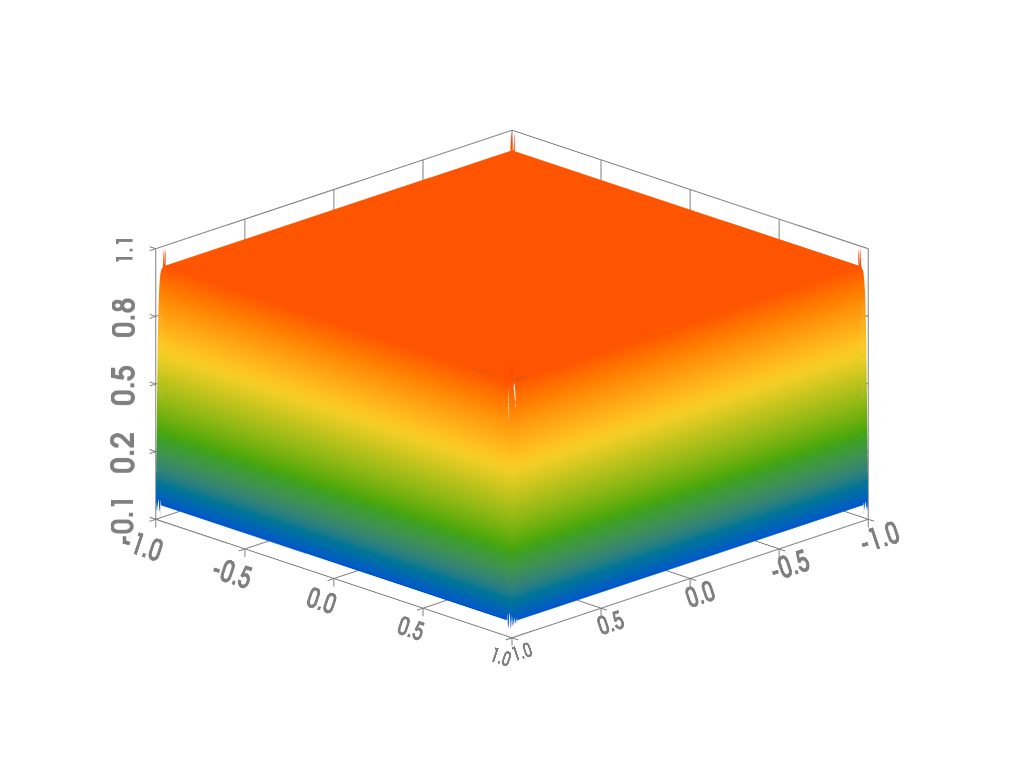}};
            \begin{scope}[x={(main.south east)},y={(main.north west)}]
                \coordinate (C) at (0.1, 0.70);
                \coordinate (T) at (-0.05, 0.71);
                \draw[black, line width=1pt] (C) circle [radius=0.04];
                \draw[black, line width=1pt] ($(C)!0.24!(T)$) -- (T);
            \end{scope}
        \node[anchor=south west, xshift=-1.6cm, yshift=3.4cm] (zoom)
          {\includegraphics[width=1.4cm, trim={17.5cm 12.9cm 15.2cm 11.5cm}, clip]{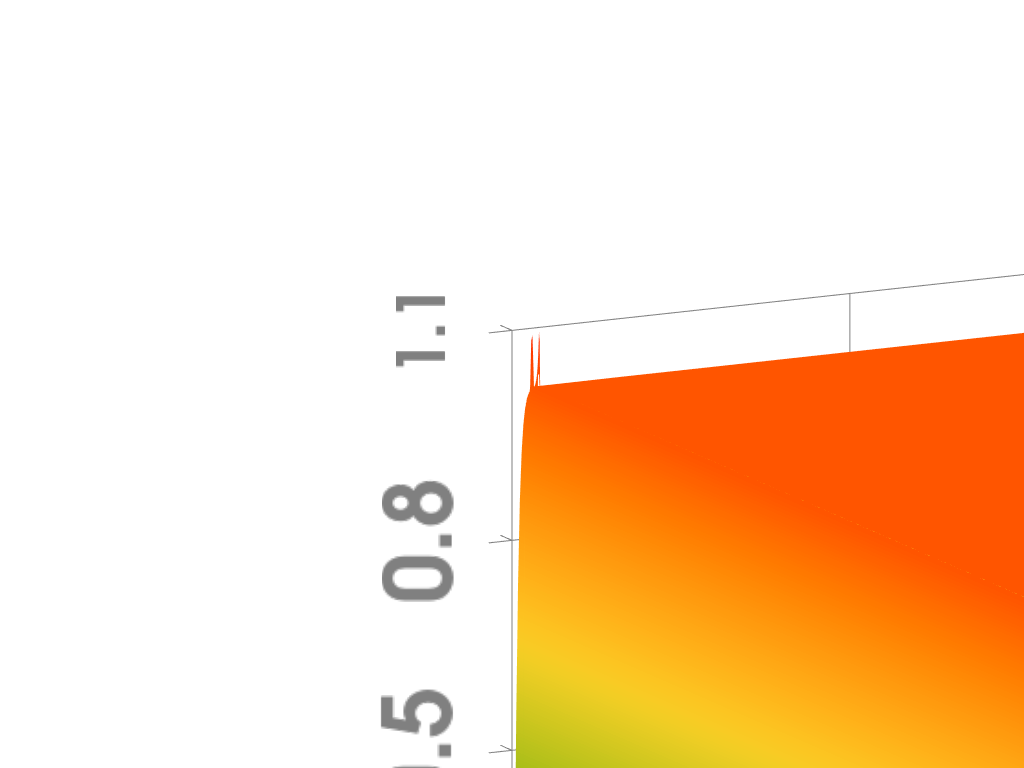}};
    \end{tikzpicture}
      \includegraphics[width=0.4\textwidth, trim={3cm 0cm 3cm 2cm}, clip]{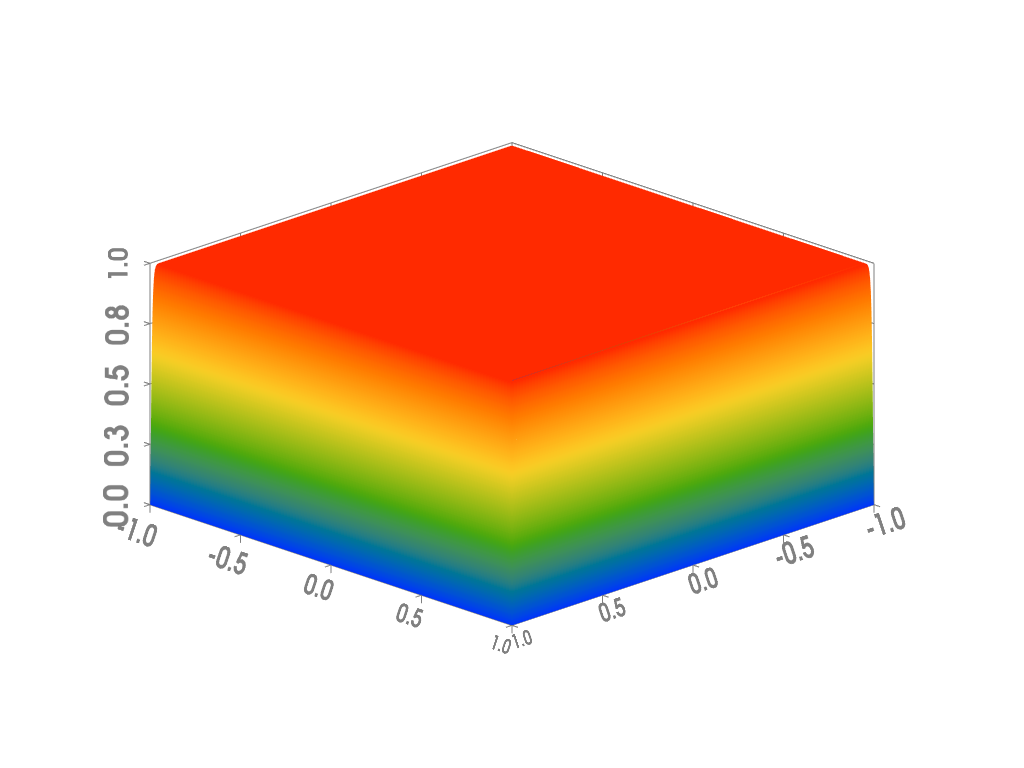}
    \vspace{-2.5em}
      \caption{Numerical solution using the embedded Trefftz DG method with $p=10$ and $
      	\lambda=0.9$ on the the nine-element mesh.
          On the left, we show the solution obtained using the same embedded Trefftz space in all elements, while on the right, we show the solution using the mixed choice of embedded Trefftz spaces.
      }
      \label{fig:9elemsol}
\vspace{-.5em}
\end{figure}

\subsection{Anisotropic diffusion problem}\label{sec:anisotropicdiffusion}
As a final example, we present an anisotropic diffusion problem  without a reaction term.  
The goal is to assess the performance of the proposed method
for problems beyond the theoretical analysis presented in this paper.
We consider the following boundary value problem:
\begin{eqs}\label{eq:diffmatrix}
	-u_{xx}- \eps^2 u_{yy} &= f\quad &&\text{ in } \Omega=(0,1)^2, \\
	u &= g \quad &&\text{ on } \partial\Omega,
\end{eqs}
	where $f=-\left(x^2+\frac{y^2}{\eps^2}\right)^{-\frac12}$, and $g$ is chosen such that the exact solution is
\begin{equation}\label{eq:solmatrixdiff}
	u(x,y)=\left(x^2+\frac{y^2}{\eps^2}\right)^\frac12.
\end{equation}
The diffusion parameter is set to $\eps^2=0.1$. 
Due to the singularities of $u_{xx}$ and $u_{yy}$ at the origin, the solution belongs only to $H^{5/2-\delta}(\Omega)$, for all $\delta>0$.
To resolve the singularity, we employ an $hp$-refinement technique similar to the one presented in \cite[\S 4.3.5]{G03dm}, for the same model problem.

\begin{figure}[ht!]
\centering
\begin{minipage}{0.45\textwidth}
  \includegraphics[width=\textwidth, trim={3cm 4cm 3cm 4cm}, clip]{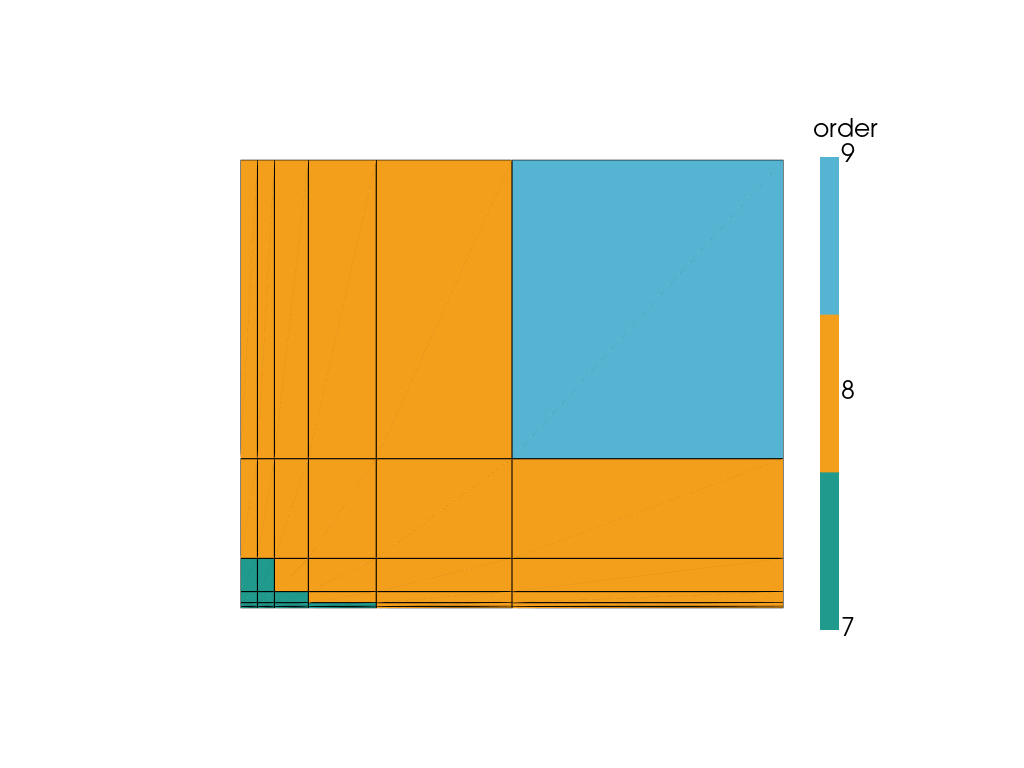}
\end{minipage}
\begin{minipage}{0.36\textwidth}
   \resizebox{\textwidth}{!}{
	\begin{tikzpicture}
		\begin{groupplot}[%
			group style={%
				group name={my plots},
				group size=2 by 2,
				horizontal sep=6em,
				vertical sep=5em,
			},
			legend style={
				legend columns=1,
				at={(0.98,0.98)},
			},
			ymajorgrids=true,
			grid style=dashed,
			cycle list name=colorshp,
			]
			\nextgroupplot[ymode=log, ylabel={$\norm{u-u_h}_{L^2(\Omega)}$},xlabel={$\sqrt{\mathrm{N}_{\mathrm{DoFs}}}$}]
			\addplot+[discard if not={method}{dg}] table [x=ndof12, y=l2error, col sep=comma] {results/example4hp.csv};
			\addplot+[discard if not={method}{embt}] table [x=ndof12, y=l2error, col sep=comma] {results/example4hp.csv};
			\addplot+[dashed,discard if not={method}{dg}] table [x=ndof12, y=l2error, col sep=comma] {results/example4h.csv};
			\addplot+[dashed,discard if not={method}{embt}] table [x=ndof12, y=l2error, col sep=comma] {results/example4h.csv};
			\legend{$\Vh$,$\ITh$,$\Vh$,$\ITh$}
		\end{groupplot}   
	\end{tikzpicture}
 }
\end{minipage}
\caption{
    $hp$-refinement for the problem \eqref{eq:diffmatrix} with exact solution $u$ in \eqref{eq:solmatrixdiff}.
    Left: finest mesh in the $hp$-refinement strategy, with geometric refinement towards the origin.
    Right: $L^2(\Omega)$-error versus square root of the number of degrees of freedom.
    Solid lines correspond to the $hp$-refinement strategy, while dashed lines correspond to $p$-refinement on the fixed finest mesh.
    }
\label{fig:diffmatrix}
\vspace{-.5em}
\end{figure}
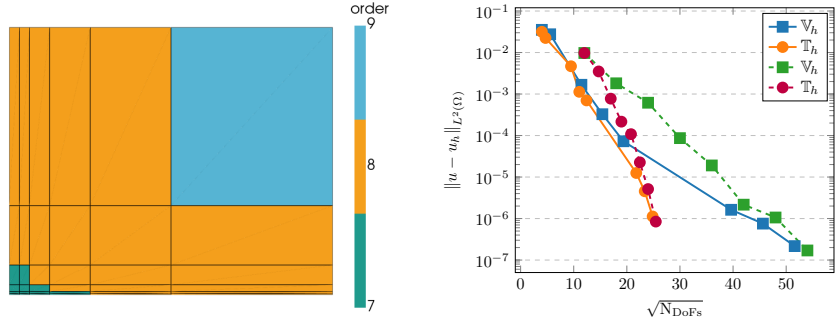

We consider a sequence of three geometrically refined meshes towards the origin.
The grading parameters along the $x$- and $y$-axis are chosen as $
\mu_x=1/2$ and $
\mu_y=1/3$, respectively. 
The finest mesh in this hierarchy is shown on the left in \cref{fig:diffmatrix}.
We compare this strategy with a pure $p$-refinement on a fixed mesh, using polynomial degrees $p = 1,\ldots,8$.
The fixed mesh is the finest of the three meshes from the $hp$-strategy, namely the one in \cref{fig:diffmatrix}.

The results, reported on the right of~\cref{fig:diffmatrix}, show the $L^2(\Omega)$-error against the square root of the number of degrees of freedom.
The $hp$-refinement strategy outperforms the $p$-refinement on the fixed mesh for both methods.
Moreover, the embedded Trefftz method consistently yields lower errors than the standard DG method when each uses its respective refinement strategy.

\section*{Acknowledgments}
This research was funded in part by the Austrian Science Fund (FWF) 
\href{https://doi.org/10.55776/ESP4389824}{10.55776/ESP4389824}.
For open access purposes, the authors have applied a CC BY public copyright license to any author-accepted manuscript version arising from this submission. SG and CP are members of the Gruppo Nazionale Calcolo
Scientifico-Istituto Nazionale di Alta Matematica (GNCS-INdAM).
CP acknowledges support from the PRIN project ``ASTICE'' (202292JW3F) funded by the European Union -- NextGenerationEU.

\sloppy
 \bibliographystyle{abbrv}
\bibliography{bib.bib}

\appendix
\section{Appendix}

	The following lemma allows us to verify the assumption~\eqref{eq:A_coerc} via a perturbation argument.
	\begin{lemma}[{\cite[Lemma 4.2]{LLSV_ARXIV_2024}}]
		\label{lem:abstractneumann}
		Let $K\in\calK$ and let $A_{K,0}:\LKh \to \LKh'$ be an invertible operator with 
		\begin{equation*}
			\norm{A_{K,0}  u_h }_{\LKh'} \geq \cAP \norm{u_h}_{{\LKh}}
			\quad \forall u_h \in \LKh,
		\end{equation*}
        for some constant~$c_{A_0} > 0$. 
        If there exists a constant $\gamma_K \in (0,1)$ such that
		\begin{equation}\label{eq:prototype}
			\norm{\AK  u _h- A_{K,0}u_h}_{\LKh'}\le\gamma_K\norm{A_{K,0}u_h}_{\LKh'} \quad \forall u_h\in \LKh,
		\end{equation}
		then $\AK : \LKh \to \LKh'$ is invertible, with
		\begin{equation*}
			\norm{A_{K}  u_h }_{\LKh'} \geq \cAP(1-\gamma_K)\norm{u_h}_{{\LKh}}
			\quad \forall u_h \in \LKh.
		\end{equation*}
 	\end{lemma}

	\begin{lemma}[Estimates for $\Pi_K^p$]\label{lem:proj_estimates}
		For any $K \in \calK$, let $v_{|_K}\in H^{\ell+1}(K)$, for some $\ell\ge 1$, and let $G_K$ be a $C^{\ell+1}$-diffeomorphism. Then, for $\widetilde v= v\circ G_K$, 
        and setting $\jj=3-i$, for $i=1,2$,
        we have  
		\begin{align}
			\|v-\Pi_K^p v \|_{K}
			&\lesssim
			\sum_{i=1}^2 h_i^{s+1}\,
			\|\partial_i^{s+1}\widetilde v\|_{\Kt},
			\label{eq:proj_est_L2}
			\\
			\|\partial_m (v-\Pi_K^p v) \|_{K}
			&\lesssim
			\sum_{i=1}^2 h_i^{s}\,
			\|\partial_i^{s+1}\widetilde v\|_{\Kt}
			+
			\sum_{i=1}^2 h_{\jj}^{s}\,
			\|\partial_{\jj}^{s}\partial_i\widetilde v\|_{\Kt},
			\qquad m=1,2,
			\label{eq:proj_est_H1}
			\\
			\|v-\Pi_K^p v\|_{\calE_i^K}
			&\lesssim
			h_{\jj}^{s+\frac12}\,
			\|\partial_{\jj}^{s+1}\widetilde v\|_{\Kt}
			+
			\Big(\frac{h_i}{h_{\jj}}\Big)^{\frac12}
			h_i^{s+
			\frac12}\,
			\|\partial_i^{s+1}\widetilde v\|_{\Kt}
			\notag\\
			&\quad
			+
			h_{\jj}^{\frac12}h_i^s\,
			\|\partial_i^s\partial_{\jj}\widetilde v\|_{\Kt},
			\qquad i=1,2,
			\label{eq:proj_est_trace}
			\\
			\|\partial_m(v-\Pi_K^p v) \|_{\calE_i^K}
			&\lesssim
			h_i^{s-\frac12}
			\Bigg[
			\Big(\frac{h_i}{h_{\jj}}\Big)^{\frac12}
			\|\partial_i^{s+1}\widetilde v\|_{\Kt}
			+
			\Big(\frac{h_{\jj}}{h_i}\Big)^{\frac12}
			\|\partial_i^s\partial_{\jj}\widetilde v\|_{\Kt}
			\Bigg]
			\notag\\
			&\quad
			+
			h_{\jj}^{s-\frac12}\,
			\|\partial_{\jj}^{s}\partial_i\widetilde v\|_{\Kt}
			+
			h_{\jj}^{s-\frac12}\,
			\|\partial_{\jj}^{s+1}\widetilde v\|_{\Kt}
			\notag\\
			&\quad
			+
			\Big(\frac{h_i}{h_{\jj}}\Big)^{\frac12}
			h_i^{s-\frac12}\,
			\|\partial_i^s\partial_{\jj}\widetilde v\|_{\Kt},
			\qquad  i,m=1,2,
			\label{eq:proj_est_dtrace}
		\end{align}
		 where $0\leq s\leq \min\{p,\ell\}$.
		The hidden constants depend only on $p$ and the geometric constants in
		\eqref{eq:GK}.
	\end{lemma}

	\end{document}